\RequirePackage{fix-cm}
\documentclass[numbook]{svjour3}                     

\smartqed  
\usepackage[T1]{fontenc}
\usepackage[utf8]{inputenc}
\usepackage{array}
\usepackage{float}
\usepackage{booktabs}
\usepackage{bm}
\usepackage{units}
\usepackage{multirow}
\usepackage{amsmath}
\usepackage{amssymb}
\usepackage{graphicx}
\usepackage{esint}
\usepackage{lineno}
\usepackage{hyperref}
\usepackage{epstopdf}
\usepackage{cite}
\usepackage{mathtools}
\usepackage{xspace}
\usepackage{enumerate}
\usepackage{subcaption}
\usepackage{nccmath, bigstrut}
\usepackage[misc]{ifsym}

\def\mycmd{1}

\newcommand{\argmax}[1]{\underset{#1}{\operatorname{arg}\,\operatorname{max}}\;}
\newcommand{\argmin}[1]{\underset{#1}{\operatorname{arg}\,\operatorname{min}}\;}
\newcommand{\mtrx}[1]{\mathbf{#1}}
\newcommand{\gmtrx}[1]{\bm{#1}}
\newcommand{\me}{\mathrm{e}}
\newcommand{\mi}{\mathrm{i}}
\newcommand{\ssm}{\textsc{ssm}tool\xspace}
\newcommand{\coco}{\textsc{coco}\xspace}
\newcommand{\matlab}{\textsc{matlab}\xspace}
\newcommand{\norm}[1]{\left\lVert#1\right\rVert}
\graphicspath{./figure/}

\makeatletter
\renewcommand*\env@matrix[1][c]{\hskip -\arraycolsep
  \let\@ifnextchar\new@ifnextchar
  \array{*\c@MaxMatrixCols #1}}
\makeatother

\journalname{Nonlinear Dynamics}
\begin{document}

\title{Exact Model Reduction and Fast Forced Response Calculation in High-Dimensional Nonlinear Mechanical Systems}
\titlerunning{Exact Model Red. and Fast FRC Calculation in High-Dim. Nonlinear Mech. Systems}   
\author{S. Ponsioen \and G. Haller}
\institute{S. Ponsioen \and G. Haller (\Letter) \at
              Institute for Mechanical Systems, 
              ETH Z{\"u}rich, 
              Leonhardstrasse 21, 
              8092, Z{\"u}rich, Switzerland \\
              \email{georgehaller@ethz.ch}
}
\date{\today}
\maketitle

\begin{abstract}
We show how spectral submanifold (SSM) theory can be used to extract forced-response curves, including isolas, without any numerical simulation in high-degree-of-freedom, periodically forced mechanical systems. We use multivariate recurrence relations to construct the SSMs,  achieving a major speed-up relative to earlier autonomous SSM algorithms. The increase in computational efficiency promises to close the current gap between studying lower-dimensional academic examples and analyzing larger systems obtained from finite-element modeling, as we illustrate on a discretization of a damped-forced beam model.  

\keywords{Spectral submanifolds \and Model-order reduction \and Forced response curves}
\end{abstract}
\section{Introduction}
Determining the forced response curve (FRC) of a multi-degree-of-freedom nonlinear mechanical system under periodic forcing is one of the most common tasks in structural engineering, providing key insights into the nonlinear behavior of the system. Specifically, the FRC gives the amplitude of the periodic response of the system as a function of the frequency of the periodic forcing. This, in turn, provides valuable information about expected material stresses and strains that arise in the system under various external forcing conditions. The nonlinear FRC often differs significantly from the FRC of the linear part of the system, possibly containing also unexpected isolated response branches (isolas).  

For low-dimensional mechanical systems, the steady-state response can simply be obtained by numerically integrating the equations of motion. However, mechanical models constructed by finite-element packages generally contain thousands of degrees of freedom. This high dimensionality, coupled with typically low damping and costly function evaluations, may result in excessively long integration times (up to days or weeks) until a steady-state response is reached. 

To overcome this obstacle, one often reduces high-dimensional systems to lower-dimensional models whose FRCs can be faster extracted. Virtually all model-reduction techniques in use involve projecting the full dynamics to a lower-dimensional subspace. Examples include the static condensation method, also known as the Guyan-Irons reduction method (Guyan \cite{guyan1965reduction} and Irons \cite{irons1965structural}; cf. G{\'e}radin and Rixen \cite{Geradin2014}), the Craig-Bampton method \cite{craig1968coupling} and the proper orthogonal decomposition method \cite{Karhunen1946, Kosambi1943, Loeve1948, Obukhov1954, Pougachev1953}. These methods are generally applied without any a priori knowledge about the errors arising from the lack of invariance of the subspace involved in the projection. Similarly unclear is the error arising from the nonlinear method of modal derivatives \cite{idelsohn1985reduction}, which formally restricts the full system into an envisioned quadratic surface in the configuration space. 
Haller and Ponsioen \cite{haller2017exact} showed that only under restrictive conditions can the static-condensation and modal-derivative techniques be justified as first- and second-order local approximations to an invariant manifold to which the full mechanical system can indeed be exactly reduced. 

A more recent reduction method, proposed by Haller and Ponsioen \cite{Haller2016}, uses spectral submanifold (SSM) theory to reduce the full dynamics to exactly invariant SSM surfaces in the phase space. SSMs are the unique, smoothest, nonlinear continuations of spectral subspaces of the linearized, unforced limit of a mechanical system. SSM theory can be applied to nonlinear, damped mechanical systems with no forcing, periodic forcing or quasi-periodic forcing. As shown by \cite{ponsioen2018automated,ponsioen2019analytic,Jain2018,Breunung2017,Szalai2017,kogelbauer2018rigorous}, the reduced dynamics on a two-dimensional SSM serves as an exact, one-degree-of-freedom reduced-order model, that can be constructed for any particular vibration mode of interest. 

Once a reduced model has been obtained by any method, it is typically interrogated for a reduced forced response. A broadly used method for this analysis is the harmonic balance (HB) method, introduced first by Kryloff and Bogoliuboff \cite{kryloff1947introduction} for a single-harmonic approximation. The HB method assumes that the system has a steady-state periodic solution, which can therefore be represented by a Fourier series. By substituting the assumed solution into the original ordinary differential equations and keeping only finitely many harmonics, one obtains a set of nonlinear algebraic equations for the unknown Fourier coefficients. The HB method can also be coupled to a continuation scheme in order to obtain the forced response over a forcing frequency domain of interest (cf. von Groll and Ewins \cite{von2001harmonic}  and Cochelin and Vergez \cite{cochelin2009high}). While conceptually simple, the HB method also has several shortcomings. First, it requires a large number of nonlinear algebraic equations to be solved, and hence becomes ineffective in higher degrees of freedom. Second, the solvability of these equations for a few harmonics does not imply that a periodic orbit actually exists. Indeed, there are documented examples of systems, such as those with quadratic nonlinearities, for which the HB has been found not to work well \cite{MickensR}. More recently, Breunung and Haller \cite{breunung2019when} constructed mechanical examples in which the HB method indicates the existence of a periodic response even though no periodic orbits exist in the system. Finally, the HB method provides no information about the stability of the periodic orbit that it approximates. 

As alternatives to the HB method, several computational methods exist in the time domain for finding periodic solutions. Among these, the shooting method (cf. Peeters et al. \cite{Peeters2009}, Slater \cite{slater1996numerical}, Roberts and Shipman \cite{roberts1972two}) solves a two-point boundary value problem to compute a steady-state solution of a periodically forced system. An initial guess, representing an initial position on the periodic orbit, is corrected by solving the equation of variations, which can be evaluated using a numerical finite-difference method by perturbing each of the initial conditions and integrating the full system. Similar to the HB method, the shooting method can be coupled to a path continuation technique to obtain the forced response curve. 

To avoid numerical integration of the full system, a collocation method can be used to solve for the full periodic solution at once. This is done by approximating a periodic solution of the full system as a continuous function of time, expressed on a predefined number of time intervals as a polynomial of a certain degree, parameterized by unknown base points (see Dankowicz and Schilder \cite{Dankowicz2013}). Collocation methods, however, have generally not been applied to large systems due to their significant memory needs.

In the recent work of Jain et al. \cite{jain2018fast}, an integral-equation approach is proposed for the fast computation of the steady-state response of (quasi-) periodically forced nonlinear systems by finding the zeros of an integral equation using a Picard and Newton–Raphson iteration method. A major advantage of this approach compared to the classical shooting method is its ability to handle quasi-periodic forcing. The integral equation approach also gives increased speed over other numerical continuation methods by exploiting the special structure of weakly nonlinear mechanical vibrations. Still, for higher degrees of freedom, even this increased speed can lead to calculations that are simply too big to be practical. 

 
In contrast to all these prior approaches, here we use the reduced dynamics on a two-dimensional SSM to extract the forced-response curve around a particular mode of interest. By doing so, we extend the work of Ponsioen et al. \cite{ponsioen2018automated}, who developed a \matlab-based computational tool (\ssm) for computing two-dimensional SSMs in arbitrary autonomous mechanical systems, to the non-autonomous setting. The present work also builds on the approach of Breunung and Haller \cite{Breunung2017}, who compute the non-autonomous part of the SSM up to zeroth order in appropriate coordinates in which the SSM-reduced dynamics simplifies to a normal form. 

The reduced dynamics on each two-dimensional SSM provides us with two differential equations. The fixed points of the two-dimensional SSM-reduced system correspond to periodic orbits on the FRC for a particular forcing frequency. These fixed points can be instantaneously computed, irrespective of the dimensionality of the original mechanical system. The stability of the corresponding periodic orbits can directly be obtained from the eigenvalues of the linearized reduced system at its fixed points. As a consequence, all periodic responses, including isolas, and their stability can be identified from a procedure in which the only numerical step in the end is the identification of the zeros of a two-dimensional autonomous vector field. A simple \matlab implementation is now available for this procedure\footnote{\ssm is available at: \url{www.georgehaller.com.}}, allowing the user to apply SSM-based model reduction and forced-response calculations to systems with high degrees of freedom. We illustrate this by locating forced responses in a forced-damped beam, considering discretizations of this nonlinear system up to 10,000 degrees of freedom. We also present speed comparisons with the collocation and the HB methods up to the limits of applicability of those methods.

\section{System set-up }
We consider $n$-degree-of-freedom, periodically forced mechanical systems of the form
\begin{gather}
\mtrx{M}\ddot{\mtrx{y}}+\mtrx{C}\dot{\mtrx{y}}+\mtrx{K}\mtrx{y}+\mtrx{g}(\mtrx{y},\dot{\mtrx{y}})=\varepsilon \mtrx{f}(\Omega t),\quad 0\leq\varepsilon\ll 1, \label{eq:mech_sys} \\ 
\mtrx{g}(\mtrx{y},\dot{\mtrx{y}})=\mathcal{O}\left(\left|\mtrx{y}\right|^{2},\left|\mtrx{y}\right|\left|\dot{\mtrx{y}}\right|,\left|\dot{\mtrx{y}}\right|^{2}\right), \nonumber 
\end{gather}
where $\mtrx{y}\in\mathbb{R}^{n}$ is the generalized position vector; $\mtrx{M}=\mtrx{M}^{T}\in\mathbb{R}^{n\times n}$ is the positive definite mass matrix; $\mtrx{C}=\mtrx{C}^{T}\in\mathbb{R}^{n\times n}$ is the damping matrix; $\mtrx{K}=\mtrx{K}^{T}\in\mathbb{R}^{n\times n}$ is the stiffness matrix and $\mtrx{g}(\mtrx{y},\dot{\mtrx{y}})$ contains all the nonlinear terms in the system, which are assumed to be analytic. The external forcing $\varepsilon \mtrx{f}(\Omega t)$ does not depend on the positions and velocities. 

We transform system (\ref{eq:mech_sys}) into a set of $2n$ first-order ordinary differential equations by introducing the change of variables $\mtrx{x}_{1}=\mtrx{y}$, $\mtrx{x}_{2}=\dot{\mtrx{y}}$, with $\mtrx{x}=(\mtrx{x}_{1},\mtrx{x}_{2})\in\mathbb{R}^{2n}$, which gives
\begin{align}
\dot{\mtrx{x}} & =
\left(\begin{array}{cc}
\mtrx{0} & \mtrx{I}\\
-\mtrx{M}^{-1}\mtrx{K} & -\mtrx{M}^{-1}\mtrx{C}
\end{array}\right)\mtrx{x}
+\left(\begin{array}{c}
\mtrx{0}\\
-\mtrx{M}^{-1}\mtrx{g}(\mtrx{x}_{1},\mtrx{x}_{2})
\end{array}\right)
+ \varepsilon\left(
\begin{array}{c}
\mtrx{0}\\
\mtrx{M}^{-1}\mtrx{f}(\Omega t)
\end{array}
\right) \nonumber \\
&=\mtrx{A}\mtrx{x}+\mtrx{G}_\text{p}(\mtrx{x})+
\varepsilon \mtrx{F}_\text{p}(\Omega t).\label{eq:dyn_sys}
\end{align}
System (\ref{eq:dyn_sys}) has a fixed point at $\mtrx{x}=\mtrx{0}$ under zero forcing ($\varepsilon=0$). Additionally, we observe that $\mtrx{M}^{-1}$ is well-defined because $\mtrx{M}$ is assumed positive definite. 

The linearized part of system (\ref{eq:dyn_sys}) is 
\begin{equation}
\dot{\vec{x}}=\vec{A}\vec{x},\label{eq:lin_dyn_sys}
\end{equation}
where the matrix $\vec{A}$ has $2n$ eigenvalues $\lambda_{k}\in\mathbb{C}$
for $k=1,\ldots,2n$. Counting multiplicities, we sort these eigenvalues
based on their real parts in the decreasing order

\begin{equation}
\text{\text{Re}}(\lambda_{2n})\leq\text{\text{Re}}(\lambda_{2n-1})\leq\ldots\leq\text{\text{Re}}(\lambda_{1})<0,
\end{equation}
assuming that the real part of each eigenvalue is less than zero and
hence the fixed point of Eq. (\ref{eq:lin_dyn_sys}) is asymptotically stable. We further assume that the constant matrix $\vec{A}$ is semisimple, and hence the algebraic multiplicity, $\text{alg}(\lambda_{k})$, is equal to the geometric multiplicity of each eigenvalue $\lambda_k$ of $\mtrx{A}$. We can, therefore, identify $2n$ linearly independent eigenvectors
$\vec{v}_{k}\in\mathbb{C}^{2n}$, with $k=1,\ldots,2n$, each spanning a
real eigenspace $E_{k}\subset\mathbb{R}^{2n}$ with $\text{dim}(E_{k})=2\times\text{alg}(\lambda_{k})$
in case $\text{Im}(\lambda_{k})\neq0$, or $\text{dim}(E_{k})=\text{alg}(\lambda_{k})$
in case $\text{Im}(\lambda_{k})=0$.

\section{Non-autonomous SSMs for continuous mechanical systems \label{sec:non_auto_red}}
As the matrix $\mtrx{A}$ is semisimple, the linear part of system (\ref{eq:dyn_sys}) is diagonalized by a linear change of coordinates $\mtrx{x}=\mtrx{T}\mtrx{q}$, with $\mtrx{T}=\left[\mtrx{v}_{1},\mtrx{v}_{2},\ldots,\mtrx{v}_{2n}\right]\in\mathbb{C}^{2n\times2n}$ and $\mtrx{q}\in\mathbb{C}^{2n}$, yielding
\begin{gather}
\dot{\mtrx{q}}=\underbrace{\text{diag}(\lambda_{1},\lambda_{2}\ldots,\lambda_{2n})}_{\vec{\Lambda}}\mtrx{q}+\mtrx{G}_\text{m}(\mtrx{q})+\varepsilon\mtrx{F}_\text{m}(\Omega t).\label{eq:ds_diag}
\end{gather}
We consider the two-dimensional modal subspace $\mathcal{E}=\text{span}\left\{ \vec{v}_{1},\vec{v}_{2}\right\} \subset\mathbb{C}^{2n}$ with $\vec{v}_{2}=\bar{\vec{v}}_{1}.$ The remaining linearly independent eigenvectors $\vec{v}_{3},\ldots,\vec{v}_{2n}$ span a complex subspace $\mathcal{C}\subset\mathbb{C}^{2n}$ such that the full phase space of (\ref{eq:ds_diag}) can be expressed as the direct sum
\begin{equation}
\mathbb{C}^{2n}=\mathcal{E}\oplus\mathcal{C}.
\end{equation}
We write the diagonal matrix $\vec{\Lambda}$ as
\begin{equation}
\vec{\Lambda}=\left[\begin{array}{cc}
\vec{\Lambda}_{\mathcal{E}} & 0\\
0 & \vec{\Lambda}_{\mathcal{C}}
\end{array}\right],\quad\text{Spect}\left(\vec{\Lambda}_{\mathcal{E}}\right)=\left\{ \lambda_{1},\lambda_{2}\right\} ,\quad\text{Spect}\left(\vec{\Lambda}_{\mathcal{C}}\right)=\left\{ \lambda_{3},\ldots,\lambda_{2n}\right\} ,\label{eq:lin_decomp}
\end{equation}
with $\vec{\Lambda}_{\mathcal{E}}=\text{diag}(\lambda_{1},\lambda_{2})$
and $\vec{\Lambda}_{\mathcal{C}}=\text{diag}(\lambda_{3},\ldots,\lambda_{2n})$.

Following Haller and Ponsioen \cite{Haller2016}, we now define a \textit{non-autonomous spectral submanifold} (SSM), $\mathcal{W}(\mathcal{E}, \Omega t)$, corresponding to the spectral subspace $\mathcal{E}$ of $\vec{\Lambda}$ as a two-dimensional invariant manifold of the dynamical system (\ref{eq:ds_diag}) that is $\frac{2\pi}{\Omega}$-periodic in time and
\begin{itemize}
\item [{(i)}] Perturbs smoothly from $\mathcal{E}$ at the trivial fixed point $\mtrx{q}=\mtrx{0}$ under the addition of the $\mathcal{O}(\varepsilon)$ terms in Eq. (\ref{eq:ds_diag}). 
\item [{(ii)}] Is strictly smoother than any other $\frac{2\pi}{\Omega}$-periodic invariant manifold satisfying (i).
\end{itemize}
We also define the \textit{absolute spectral quotient} $\Sigma(\mathcal{E})$ of $\mathcal{E}$ as the positive integer 
\begin{equation}
\Sigma(\mathcal{E})=\text{Int}\left[\frac{\min_{\lambda\in\text{Spect}(\vec{\Lambda})}\text{Re}\lambda}{\max_{\lambda\in\text{Spect}(\vec{\Lambda}_{\mathcal{E}})}\text{Re}\lambda}\right]\in\mathbb{N}^{+}.\label{eq:abs_spect_quo}
\end{equation}
Additionally, we introduce the non-resonance conditions 
\begin{equation}
a\text{Re}\lambda_{1}+b\text{Re}\lambda_{2}\neq\text{Re}\lambda_{l},\quad\forall\lambda_{l}\in\text{Spect}(\vec{\Lambda}_{\mathcal{C}}),\quad 2\leq a+b\leq\Sigma(\mathcal{E}),\quad a,b\in\mathbb{N}. \label{eq:ext_res}
\end{equation}
We now restate the following result from Haller and Ponsioen \cite{Haller2016} on the existence of an SMM in system (\ref{eq:ds_diag}).
\begin{theorem}\label{thrm:SSM} 
Under the non-resonance conditions (\ref{eq:ext_res}), the following hold for system (\ref{eq:ds_diag}):
\begin{itemize}
\item [{(i)}] There exists a unique two-dimensional, time-periodic, analytic SSM, $\mathcal{W}(\mathcal{E}, \Omega t)$ that depends smoothly on the parameter $\epsilon$.
\item [{(ii)}]$\mathcal{W}(\mathcal{E})$ can be viewed as an embedding of an open set $\mathcal{U}$ into the phase space of system (\ref{eq:ds_diag}) via the map
\begin{equation}
\vec{W}(\mtrx{s},\phi):\mathcal{U}\subset\mathbb{C}^{2}\times S^1\rightarrow\mathbb{C}^{2n},\label{eq:W_map}
\end{equation}
with the phase variable $\phi \in S^1$. We can approximate $\mtrx{W}(\mtrx{s},\phi)$ in a neighborhood of the origin using a Taylor expansion in the parameterization coordinates $\mtrx{s}=(s_1,s_2=\bar{s}_1)$, with coefficients that depend periodically on the phase variable $\phi$.
\item [{(iii)}] There exists a polynomial function $\vec{R}(\mtrx{s},\phi):\mathcal{U}\rightarrow\mathcal{U}$
satisfying the invariance relationship 
\upshape
\begin{equation}
\mtrx{\Lambda} \mtrx{W}(\mtrx{s},\phi)+\mtrx{G}_\text{m}(\mtrx{W}(\mtrx{s},\phi))+\varepsilon\mtrx{F}_\text{m}(\phi)=D_\mtrx{s}\mtrx{W}(\mtrx{s},\phi)\mtrx{R}(\mtrx{s},\phi)+D_{\phi}\mtrx{W}(\mtrx{s},\phi)\Omega,\label{eq:invar}
\end{equation}
\itshape
such that the reduced dynamics on the SSM can be expressed as
\begin{equation}
\dot{\mtrx{s}}=\mtrx{R}(\mtrx{s},\phi). \label{eq:map_R}
\end{equation}
\end{itemize}
\end{theorem}
\begin{proof}:
We have simply restated the main theorem by Haller and Ponsioen \cite{Haller2016}, which is based on the more abstract results of Cabré et al. \cite{Cabre2003,Cabre2003b,Cabre2005} for mappings on Banach spaces. \qed
\end{proof}
In the upcoming sections, we will explain how to construct non-autonomous SSMs and show that the fixed points of the reduced dynamics represent limit cycles in the full phase space. These limit cycles, in turn, each correspond to points on the FRC for a particular forcing frequency. 

\section{Non-autonomous SSM computation \label{app:construct_ssm}}
By the smooth dependence of the SSM on $\varepsilon$, we can write
\begin{align}
\mtrx{W}(\mtrx{s},\phi)&=\mtrx{W}_0(\mtrx{s})+\varepsilon\mtrx{W}_1(\mtrx{s},\phi)+\mathcal{O}(\varepsilon^2), \label{eq:SSM_exp}\\
\mtrx{R}(\mtrx{s},\phi)&=\mtrx{R}_0(\mtrx{s})+\varepsilon\mtrx{R}_1(\mtrx{s},\phi)+\mathcal{O}(\varepsilon^2).\label{eq:RED_exp}
\end{align}
We now substitute Eqs. (\ref{eq:SSM_exp})-(\ref{eq:RED_exp}) into the invariance Eq. (\ref{eq:invar}) and collect terms of equal order in $\varepsilon$. Given that $\mtrx{G}_\text{m}(\mtrx{q})=\mathcal{O}(\left|\mtrx{q}\right|^2)$, we can Taylor-expand $\mtrx{G}_\text{m}(\mtrx{W}(\mtrx{s},\phi))$ around $\varepsilon=0$, to obtain
\begin{equation}
\mtrx{G}_\text{m}(\mtrx{W}(\mtrx{s},\phi)) = \mtrx{G}_\text{m}(\mtrx{W}_0(\mtrx{s}))+\varepsilon D_\mtrx{q}\mtrx{G}_\text{m}(\mtrx{W}_0(\mtrx{s}))\mtrx{W}_1(\mtrx{s},\phi)+\mathcal{O}(\varepsilon^2).
\end{equation}
\subsection{The autonomous coefficient equations}
Collecting terms of $\mathcal{O}(1)$ in Eq. (\ref{eq:invar}), we obtain the coefficient equations for the autonomous part of the SSM:
\begin{equation}
\gmtrx{\Lambda} \mtrx{W}_0(\mtrx{s}) + \mtrx{G}_\text{m}(\mtrx{W}_0(\mtrx{s})) = D_\mtrx{s} \mtrx{W}_0(\mtrx{s})\mtrx{R}_0(\mtrx{s}).\label{eq:auto_SSM}
\end{equation}
The autonomous part of the SSM and the reduced dynamics, which have previously been derived from an expansion in $\varepsilon$, are in turn Taylor expanded in the parameterization coordinates $\mtrx{s}$, which we explicitly express as 
\begin{gather}
\mtrx{W}_0(\mtrx{s}) =
\begin{bmatrix}
w_1^0(\mtrx{s}) \\
\vdots \\
w_{2n}^0(\mtrx{s}) 
\end{bmatrix},\quad w_i^0(\mtrx{s})=\sum_{\mtrx{m}}W_{i,\mtrx{m}}^0\mtrx{s}^\mtrx{m},  \\
\mtrx{R}_0(\mtrx{s}) =
\begin{bmatrix}
r_1^0(\mtrx{s}) \\
r_2^0(\mtrx{s}) 
\end{bmatrix},\quad r_i^0(\mtrx{s})=\sum_{\mtrx{m}}R_{i,\mtrx{m}}^0\mtrx{s}^\mtrx{m},
\end{gather}
with the multi-index notation $\mtrx{m}\in\mathbb{N}_0^{2}$. 
\begin{theorem}\label{thrm:coef_eq_auto}
The coefficient equation related to the $\textbf{k}^\text{th}$-power term of the $i^\text{th}$ row of the autonomous invariance Eq. (\ref{eq:auto_SSM}), for $|\mtrx{k}|>2$, is equal to
\begin{equation}
\left(\lambda_i-\sum_{j=1}^{2}k_j\lambda_j\right)W_{i,\mtrx{k}}^0 = \sum_{j=1}^{2}\delta_{ij}R_{j,\mtrx{k}}^0 + Q_{i,\mtrx{k}}, \label{eq:coef_EQ_W0}
\end{equation}
where $Q_{i,\mtrx{k}}$ can be written as 
\begin{align}
& \sum_{j=1}^{2}\sum_{\substack{\mtrx{m}\leq\tilde{\mtrx{k}}_j \\ \mtrx{m}\neq \mtrx{e}_j\\ \mtrx{m}\neq\mtrx{k}_{\text{ }} \\ m_j>0}} m_j W_{i,\mtrx{m}}^0 R_{j,\tilde{\mtrx{k}}_j-\mtrx{m}}^0-\left[g_i(\mtrx{W}_0(\mtrx{s}))\right]_\mtrx{k}. \nonumber
\end{align}
\end{theorem}
\begin{proof}
We derive this result in Appendix \ref{app:coef_eq_auto}.\qed
\end{proof}
\subsubsection{Solving the autonomous invariance equation for $|\mtrx{k}|>0$}
As the autonomous part of the SSM is tangent to the spectral subspace $\mathcal{E}$ by construction (see Cabr{\'e} et al. \cite{Cabre2005}), we have that
\begin{gather}
\mtrx{W}_0(\mtrx{0})=\mtrx{0},\quad D_\mtrx{s}\mtrx{W}_0(\mtrx{0})\mathcal{E}=\mathcal{E}, \nonumber\\
\mtrx{R}_0(\mtrx{0})=\mtrx{0},\quad D_\mtrx{s}\mtrx{R}_0(\mtrx{0})=\Lambda_\mathcal{E}, \nonumber
\end{gather}
which satisfies the autonomous coefficient Eq. (\ref{eq:auto_SSM}) for $|\mtrx{k}|=0$ and $|\mtrx{k}|=1$. For $|\mtrx{k}|\geq 2$, we solve Eq. (\ref{eq:coef_EQ_W0}) for $W_{i,\mtrx{k}}^0$, which yields
\begin{equation}
W_{i,\mtrx{k}}^0 = \frac{\sum_{j=1}^{2}\delta_{ij}R_{j,\mtrx{k}}^0 + Q_{i,\mtrx{k}}}{\lambda_i-\sum_{j=1}^{2}k_j\lambda_j}. \label{eq:auto_sol}
\end{equation}
\subsection{Removing near-resonant terms from the autonomous SSM}
As observed by Szalai et al. \cite{Szalai2017}, if the spectral subspace $\mathcal{E}$ is lightly damped, the near-resonance relationships 
\begin{equation}
\lambda_1 - \left((k+1)\lambda_1 + k\lambda_2\right)\approx 0, \quad \lambda_2 - \left(k\lambda_1 + (k+1)\lambda_2\right) \approx 0
\end{equation}
hold for $k\in\mathbb{N}^+$. Specifically, we consider the damping in the spectral subspace $\mathcal{E}$ light if
\begin{equation}
\left|\text{Re}(\lambda_1)\right| \ll \frac{1}{2k}. \label{eq:small_damp_E}
\end{equation}
When this relation holds,  Eq. (\ref{eq:auto_sol}) will have large denominators, generally reducing the domain of convergence of the Taylor series approximations for $\mtrx{W}(\mtrx{s})$. To counter this effect, we will remove these near-resonant terms from the expression of the autonomous SSM and place them in the autonomous part of the reduced dynamics by setting 
\begin{gather}
R_{1,(k+1,k)}^0 = -Q_{1,(k+1,k)}:=\gamma_k \quad \Rightarrow  \quad W_{1,(k+1,k)}^0=0, \\
R_{2,(k,k+1)}^0 = -Q_{2,(k,k+1)}:=\bar{\gamma}_k \quad \Rightarrow  \quad W_{2,(k,k+1)}^0=0.
\end{gather}
This results in
\begin{equation}
\mtrx{R}_0(\mtrx{s}) = 
\begin{bmatrix}
\lambda_1 s_1 + \sum_{i=1}^M\gamma_is_1^{i+1}\bar{s}_1^{i} \\
\bar{\lambda}_1 \bar{s}_1 + \sum_{i=1}^M\bar{\gamma}_is_1^{i}\bar{s}_1^{i+1}
\end{bmatrix},\quad M\in\mathbb{N}^+,
\end{equation}
where we assumed that
\begin{equation}
\left|\text{Re}(\lambda_1)\right| \ll \frac{1}{2M}. \label{eq:order_small_auto}
\end{equation}

\subsection{The non-autonomous coefficient equations}
Collecting terms of $\mathcal{O}(\varepsilon)$ in Eq. (\ref{eq:invar}), we obtain 
\begin{align}
&\gmtrx{\Lambda} \mtrx{W}_1(\mtrx{s},\phi) + D_\mtrx{q}\mtrx{G}_\text{m}(\mtrx{W}_0(\mtrx{s}))\mtrx{W}_1(\mtrx{s},\phi) + \mtrx{F}_\text{m}(\phi) \label{eq:invar_O_eps}\\ & = D_\mtrx{s} \mtrx{W}_0(\mtrx{s})\mtrx{R}_1(\mtrx{s},\phi) 
+ D_\mtrx{s} \mtrx{W}_1(\mtrx{s},\phi)\mtrx{R}_0(\mtrx{s}) + D_{\phi} \mtrx{W}_1(\mtrx{s},\phi)\Omega. \nonumber
\end{align}
The non-autonomous part of the SSM and the reduced dynamics, are Taylor-expanded in the parameterization coordinates $\mtrx{s}$, which we explicitly express as 
\begin{gather}
\mtrx{W}_1(\mtrx{s},\phi) =
\begin{bmatrix}
w_1^1(\mtrx{s},\phi) \\
\vdots \\
w_{2n}^1(\mtrx{s},\phi) 
\end{bmatrix},\quad w_i^1(\mtrx{s},\phi)=\sum_{\mtrx{m}}W_{i,\mtrx{m}}^1(\phi)\mtrx{s}^\mtrx{m}, \\
\mtrx{R}_1(\mtrx{s},\phi) =
\begin{bmatrix}
r_1^1(\mtrx{s},\phi) \\
r_2^1(\mtrx{s},\phi) 
\end{bmatrix},\quad r_i^1(\mtrx{s},\phi)=\sum_{\mtrx{m}}R_{i,\mtrx{m}}^1(\phi)\mtrx{s}^\mtrx{m}. 
\end{gather}
\begin{theorem}\label{thrm:coef_eq}
For $\phi\in S^1$, the coefficient equation related to the $\textbf{k}^\text{th}$-power term of the $i^\text{th}$ row of the non-autonomous invariance Eq. (\ref{eq:invar_O_eps}) is equal to
\begin{equation}
\left(\lambda_i-\sum_{j=1}^{2}k_j\lambda_j\right)W_{i,\mtrx{k}}^1(\phi)-D_{\phi} W_{i,\mtrx{k}}^1(\phi)\Omega = \sum_{j=1}^{2}\delta_{ij}R_{j,\mtrx{k}}^1(\phi) + P_{i,\mtrx{k}}(\phi), \label{eq:coef_EQ_W1}
\end{equation}
where $P_{i,\mtrx{k}}(\phi)$ can be written as 
\begin{align}
P_{i,\mtrx{k}}(\phi) &= \sum_{j=1}^{2}\sum_{\substack{\mtrx{m}\leq\tilde{\mtrx{k}}_j \\ \mtrx{m}\neq \mtrx{e}_j \\ m_j>0}} m_j W_{i,\mtrx{m}}^0 R_{j,\tilde{\mtrx{k}}_j-\mtrx{m}}^1(\phi)+\sum_{j=1}^{2}\sum_{\substack{\mtrx{m}\leq\tilde{\mtrx{k}}_j \\ \mtrx{m}\neq\mtrx{k} \\ m_j>0}} m_j W_{i,\mtrx{m}}^1(\phi)R_{j,\tilde{\mtrx{k}}_j-\mtrx{m}}^0 \label{eq:def_P}\\
&-F_{i,\mtrx{k}}(\phi) -\left[\sum_{j=1}^{2n}D_{q_j}g_i(\mtrx{W}_0(\mtrx{s}))w_j^1(\mtrx{s},\phi)\right]_\mtrx{k}. \nonumber
\end{align}
\end{theorem}
\begin{proof}
We derive this result in Appendix \ref{app:coef_eq}.\qed
\end{proof}
\subsubsection{Solving the non-autonomous invariance equation for $|\mtrx{k}|=0$ \label{sec:sol_k_0}}
For $|\mtrx{k}|=\mtrx{0}$, Eq. (\ref{eq:coef_EQ_W1}) becomes
\begin{equation}
\lambda_iW_{i,\mtrx{0}}^1(\phi)-D_{\phi} W_{i,\mtrx{0}}^1(\phi)\Omega = \sum_{j=1}^{2}\delta_{ij}R_{j,\mtrx{0}}^1(\phi) -F_{i,\mtrx{0}}(\phi). \label{eq:coef_EQ_W1_k_0}
\end{equation}
Assuming that the forcing term $F_{i,\mtrx{0}}(\phi)$ can be written as 
\begin{equation}
F_{i,\mtrx{0}}(\phi)= \tilde{F}_{i,\mtrx{0}}\frac{\me^{\mi\phi}+\me^{-\mi\phi}}{2},
\end{equation}
we express $W_{i,\mtrx{0}}^1(\phi)$ and $R_{i,\mtrx{0}}^1(\phi)$ in the following form
\begin{equation}
W_{i,\mtrx{0}}^1(\phi) = a_{i,\mtrx{0}}\me^{\mi\phi} + b_{i,\mtrx{0}}\me^{-\mi\phi},\quad
R_{i,\mtrx{0}}^1(\phi) = c_{i,\mtrx{0}}\me^{\mi\phi} + d_{i,\mtrx{0}}\me^{-\mi\phi}.
\end{equation}
We can now write the solution of Eq. (\ref{eq:coef_EQ_W1_k_0}) as 
\begin{equation}
W_{i,\mtrx{0}}^1 = \frac{\delta_{i1}c_{1,\mtrx{0}}+\delta_{i2}c_{2,\mtrx{0}}-\frac{1}{2}\tilde{F}_{i,\mtrx{0}}}{\lambda_i-\mi\Omega}\me^{\mi\phi} + \frac{\delta_{i1}d_{1,\mtrx{0}}+\delta_{i2}d_{2,\mtrx{0}}-\frac{1}{2}\tilde{F}_{i,\mtrx{0}}}{\lambda_i+\mi\Omega}\me^{-\mi\phi}. \label{eq:coef_0}
\end{equation}
For lightly damped systems where $\text{Re}\lambda_1$ is small, we obtain small denominators in Eq. (\ref{eq:coef_0}) if the forcing frequency $\Omega$ is approximately equal to $\text{Im}\lambda_1$. We, therefore, intend to remove this near-resonance by setting
\begin{equation}
c_{1,\mtrx{0}}=\frac{1}{2}\tilde{F}_{1,\mtrx{0}},\quad c_{2,\mtrx{0}}=0,\quad
d_{1,\mtrx{0}}=0,\quad d_{2,\mtrx{0}}=\frac{1}{2}\tilde{F}_{2,\mtrx{0}}. \label{eq:c0}
\end{equation}
\subsubsection{Solving the non-autonomous invariance equation for $|\mtrx{k}|>0$}
For $|\mtrx{k}|>0$, the solution to the non-autonomous invariance Eq. (\ref{eq:coef_EQ_W1}) takes the form
\begin{equation}
W_{i,\mtrx{k}}^1(\phi)=\underbrace{\frac{\sum_{j=1}^{2}\delta_{ij}c_{j,\mtrx{k}}+\alpha_{i,\mtrx{k}}}{\lambda_i-\sum_{j=1}^{2}k_j\lambda_j-\mi\Omega}}_{a_{i,\mtrx{k}}}\me^{\mi\phi}+\underbrace{\frac{\sum_{j=1}^{2}\delta_{ij}d_{j,\mtrx{k}}+\beta_{i,\mtrx{k}}}{\lambda_i-\sum_{j=1}^{2}k_j\lambda_j+\mi\Omega}}_{b_{i,\mtrx{k}}}\me^{-\mi\phi}, \label{eq:coef_EQ_W1_k_higher}
\end{equation}
where we introduced the following notation for $P_{i,\mtrx{k}}$ in Eq. (\ref{eq:def_P}) 
\begin{equation}
P_{i,\mtrx{k}}=\alpha_{i,\mtrx{k}}\me^{\mi \phi}+\beta_{i,\mtrx{k}}\me^{-\mi \phi} . \nonumber
\end{equation}
\subsection{Removing near-resonant terms from the non-autonomous SSM}
Using the same reasoning as in section \ref{sec:sol_k_0}, we want to choose $c_{i,\mtrx{k}}$ and $d_{i,\mtrx{k}}$ in Eq. (\ref{eq:coef_EQ_W1_k_higher}) in a way to prevent the coefficients $a_{i,\mtrx{k}}$ and $b_{i,\mtrx{k}}$ from having any small denominators. We observe that if the spectral subspace $\mathcal{E}$ is lightly damped and the forcing frequency $\Omega$ is close to $\text{Im}\lambda_1$, the near-resonance relationships
\begin{align}
&\lambda_1 -\left(k\lambda_1 + k\lambda_2\right)-\mi\Omega\approx 0, \nonumber\\
&\lambda_1 -\left((k+1)\lambda_1 + (k-1)\lambda_2\right)+\mi\Omega\approx 0, \nonumber\\
&\lambda_2 -\left(k\lambda_1 + k\lambda_2\right)+\mi\Omega\approx 0, \nonumber\\
&\lambda_2 -\left((k-1)\lambda_1 + (k+1)\lambda_2\right)-\mi\Omega\approx 0, \nonumber
\end{align}
hold for $k\in\mathbb{N}^+$, where, for the non-autonomous expressions, a lightly damped spectral subspace $\mathcal{E}$ implies that 
\begin{equation}
\left|\text{Re}(\lambda_1)\right| \ll \frac{1}{\left|1-2k\right|}. \label{eq:small_damp_E_non} 
\end{equation}
Eq. (\ref{eq:small_damp_E_non}) is automatically satisfied if the small damping assumption in Eq. (\ref{eq:small_damp_E}) is satisfied, because
\begin{equation}
\frac{1}{2k}<\frac{1}{\left|1-2k\right|},\quad k\in\mathbb{N}^+.
\end{equation}   
The near-resonance terms are removed from the expressions of $\mtrx{W}_1(\mtrx{s},\phi)$ and included into the non-autonomous part of the reduced dynamics $\mtrx{R}_1(\mtrx{s},\phi)$ if we set
\begin{align}
c_{1,(k,k)}=-\alpha_{1,(k,k)}\quad &\Rightarrow \quad a_{1,(k,k)}=0, \nonumber \\
d_{2,(k,k)}=-\beta_{2,(k,k)}\quad &\Rightarrow \quad b_{2,(k,k)}=0, \nonumber \\
d_{1,(k+1,k-1)}=-\beta_{1,(k+1,k-1)}\quad &\Rightarrow \quad b_{1,(k+1,k-1)}=0, \nonumber \\
c_{2,(k-1,k+1)}=-\alpha_{2,(k-1,k+1)}\quad &\Rightarrow \quad a_{2,(k-1,k+1)}=0, \nonumber
\end{align}
where, by construction, we have
\begin{align}
d_{2,(k,k)} &= \bar{c}_{1,(k,k)}, \nonumber \\
c_{2,(k-1,k+1)} &= \bar{d}_{1,(k+1,k-1)}. \nonumber
\end{align}
This results in the following form for the non-autonomous part of the reduced dynamics:
\begin{equation}
\mtrx{R}_1(\mtrx{s},\phi)=
\begin{bmatrix}
c_{1,\mtrx{0}}\me^{\mi\phi} &+ \sum_{i=1}^M\left(c_{1,(i,i)}(\Omega)s_1^i\bar{s}_1^i\me^{\mi\phi}+d_{1,(i+1,i-1)}(\Omega)s_1^{i+1}\bar{s}_1^{i-1}\me^{-\mi\phi}\right) \\
\bar{c}_{1,\mtrx{0}}\me^{-\mi\phi} &+ \sum_{i=1}^M\left(\bar{c}_{1,(i,i)}(\Omega)s_1^i\bar{s}_1^i\me^{-\mi\phi}+\bar{d}_{1,(i+1,i-1)}(\Omega)s_1^{i-1}\bar{s}_1^{i+1}\me^{\mi\phi}\right)
\end{bmatrix},\nonumber
\end{equation}
where Eq. (\ref{eq:order_small_auto}) implies that $\left|\text{Re}(\lambda_1)\right| \ll \frac{1}{2M} < \frac{1}{\left|1-2M\right|}$.

\section{Reduced dynamics on the non-autonomous SSM}
Our next result concerns the dynamics on the SSM described in Theorem \ref{thrm:SSM}
\begin{theorem}\label{thm:red_dyn}
Under the assumption that $\left|\text{Re}(\lambda_1)\right| \ll \frac{1}{2M}$, the dynamics on the two-dimensional SSM given in Theorem \ref{thrm:SSM} can approximately be written in polar coordinates $(\rho,\psi)$ as
\begin{align}
\dot{\rho}&=a(\rho)+\varepsilon\left(f_1(\rho,\Omega)\cos(\psi)+f_2(\rho,\Omega)\sin(\psi)\right),\label{eq:red1_orig} \\
\dot{\psi}&=(b(\rho)-\Omega)+\frac{\varepsilon}{\rho}\left(g_1(\rho,\Omega)\cos(\psi)-g_2(\rho,\Omega)\sin(\psi)\right),\label{eq:red2_orig}
\end{align}
where
\upshape
\begin{align}
a(\rho)   &= \text{Re}(\lambda_1)\rho+\sum_{i=1}^M\text{Re}(\gamma_i)\rho^{2i+1},\label{eq:auto_a} \\
b(\rho)   &= \text{Im}(\lambda_1)+\sum_{i=1}^M\text{Im}(\gamma_i)\rho^{2i},  \\
f_1(\rho,\Omega) &= \text{Re}(c_{1,\mtrx{0}}) + \sum_{i=1}^M\left(\text{Re}(c_{1,(i,i)}(\Omega))+\text{Re}(d_{1,(i+1,i-1)}(\Omega))\right)\rho^{2i}, \\
f_2(\rho,\Omega) &= \text{Im}(c_{1,\mtrx{0}}) + \sum_{i=1}^M\left(\text{Im}(c_{1,(i,i)}(\Omega))-\text{Im}(d_{1,(i+1,i-1)}(\Omega))\right)\rho^{2i}, \\
g_1(\rho,\Omega) &= \text{Im}(c_{1,\mtrx{0}}) + \sum_{i=1}^M\left(\text{Im}(c_{1,(i,i)}(\Omega))+\text{Im}(d_{1,(i+1,i-1)}(\Omega))\right)\rho^{2i}, \\
g_2(\rho,\Omega) &= \text{Re}(c_{1,\mtrx{0}}) + \sum_{i=1}^M\left(\text{Re}(c_{1,(i,i)}(\Omega))-\text{Re}(d_{1,(i+1,i-1)}(\Omega))\right)\rho^{2i}, \label{eq:g2}
\end{align}
\itshape
with $2M+1$ denoting the order of the expansion.
\end{theorem}
\begin{proof}:
We derive this result in Appendix \ref{app:red_dyn}. \qed
\end{proof}
We note that Theorem \ref{thrm:SSM}, upon which Theorem \ref{thm:red_dyn} is based, is specifically geared towards constructing the SSM corresponding to the slowest vibration mode of system (\ref{eq:ds_diag}). However, the main result of Haller and Ponsioen \cite{Haller2016} is general enough to allow for the construction of an SSM over any mode of interest as long as appropriate non-resonance conditions are satisfied. Therefore, an approach similar to the one described in this section can be applied to extract the FRCs of higher-order modes.  

In the unforced limit ($\varepsilon=0$), the reduced system (\ref{eq:red1_orig})-(\ref{eq:red2_orig}) can have fixed points but no nontrivial periodic orbits. This is because (\ref{eq:red1_orig}) decouples from (\ref{eq:red2_orig}), representing a one-dimensional ordinary differential equation that cannot have non-constant periodic solutions. By construction, the trivial fixed point of (\ref{eq:red1_orig})-(\ref{eq:red2_orig}) is asymptotically stable and will persist for $\varepsilon>0$. These persisting fixed points satisfy the system of equations
\begin{equation}
\mtrx{F}(\mtrx{u}) = 
\begin{bmatrix}
F_1(\mtrx{u})\\
F_2(\mtrx{u})
\end{bmatrix} = 
\begin{bmatrix}
a(\rho) + \varepsilon\left(f_1(\rho,\Omega)\cos(\psi)+f_2(\rho,\Omega)\sin(\psi)\right) \\
(b(\rho)-\Omega)\rho + \varepsilon\left(g_1(\rho,\Omega)\cos(\psi)-g_2(\rho,\Omega)\sin(\psi)\right) 
\end{bmatrix} = \mtrx{0}, \label{eq:zeroproblem}\\
\end{equation}
where
\begin{equation}
\mtrx{F}(\mtrx{u}):\mathbb{R}^3 \rightarrow \mathbb{R}^2, \quad
\mtrx{u} = 
\begin{bmatrix}
\rho \\
\Omega\\ 
\psi
\end{bmatrix}. \nonumber
\end{equation}
If there exists a regular point  $\mtrx p=(\rho,\Omega,\psi)$, such that $\mtrx F(\mtrx p )=\mtrx 0$  in (\ref{eq:zeroproblem}) and the Jacobian of $\mtrx F$ evaluated at $\mtrx p$ is surjective, then by the implicit function theorem, locally there exists a one-dimensional submanifold of $\mathbb{R}^3$ which will represent the forced response curve when projected onto the $(\Omega,\rho)$-space. 
The stability of these fixed points (which correspond to periodic solutions of the full mechanical system) is determined by the real parts of the eigenvalues of the Jacobian of $\mtrx{F}(\mtrx{u})$, as illustrated in Fig. \ref{fig:overview}.

\begin{figure}[h!]		
\centering		
    \begin{subfigure}{1\textwidth}		
        \centering		
        \if\mycmd1		
        \includegraphics[scale=0.45]{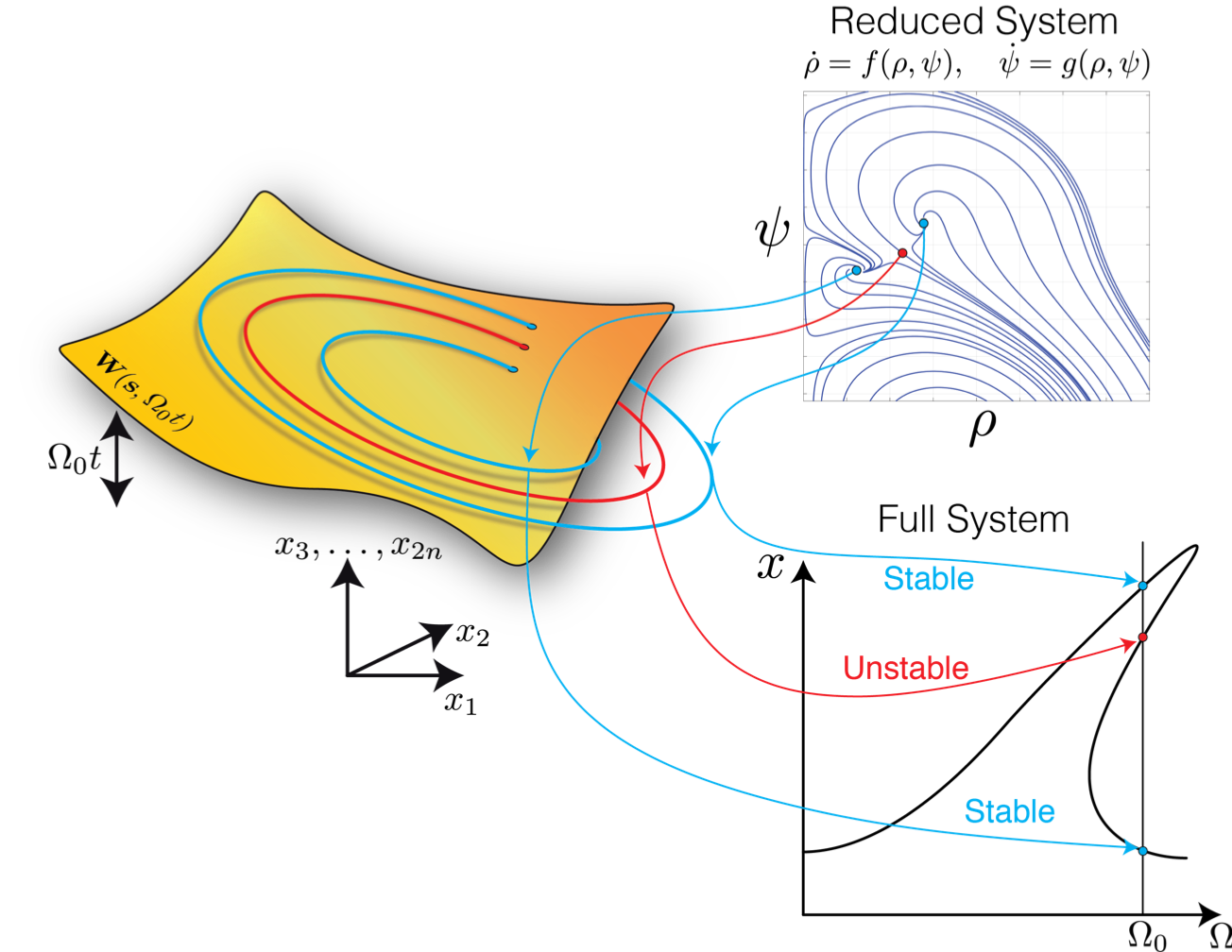}		
        \fi		
    \end{subfigure}		
    \caption{Illustration of how the fixed points of the reduced dynamics for a fixed forcing frequency $\Omega_0$ are mapped to periodic orbits in the full phase space by the mapping $\mtrx{W}(\mtrx{s},\Omega_0 t)$.  \label{fig:overview}}		
\end{figure}		

\noindent In Appendix \ref{app:geometric}, we give a geometric interpretation of the construction of zeros for the reduced dynamics on the SSM.

In summary, Theorem \ref{thm:red_dyn} gives explicit formulas that enable the calculation of the exact dynamics up to any required order of accuracy for the SSMs associated with the normal modes of the original mechanical system (\ref{eq:mech_sys}). Once the reduced dynamics is calculated, finding the nonlinear periodic responses of the system, including isolas, simply amounts to finding the zeros of the right-hand side of Eqs. (\ref{eq:red1_orig})-(\ref{eq:red2_orig}). No other numerical simulation or iteration is involved in constructing the forced response from SSM-based, exact model reduction.

\section{Example: A discretized, forced Bernoulli beam with a nonlinear spring \label{sec:beam}}
As an application of our main result on non-autonomous, SSM-based model reduction and forced response, we now consider a discretized, cantilevered Bernoulli beam with a cubic spring attached to the free end of the beam. We extract the forced-response curve around the first eigenfrequency of the beam using \ssm\footnote{\ssm is available at: \url{www.georgehaller.com.}}, the HB method (\textsc{nl}vib tool \cite{Malte2018}) and the \texttt{po} toolbox of \textsc{coco}, a numerical continuation package discussed in \cite{Dankowicz2013}. We apply all three methods on the same discretized beam for an increasing number of elements in the discretization, ranging from 10 degrees of freedom to 10,000 degrees of freedom. We note that \textsc{nl}vib tool and \textsc{coco} only run in series. Indeed, neither approach would benefit from parallelization over different forcing cases, as steady-state responses forced for one parameter configuration are heavily used to initialize the search for steady states for the next parameter configuration. In contrast, finding steady states from \ssm involves no numerical simulations or iterations and hence can be done in parallel for all forcing parameter values of interest. We will nevertheless include results from \ssm run in series, in addition to a parallelized run over 20 processors. 

The beam is of length $L$, with the square cross-section $A$, situated in a Cartesian coordinate system of $(x,y,z)$ and basis $(\mtrx{e}_x,\mtrx{e}_y,\mtrx{e}_z)$. The relevant beam parameters are listed in Table \ref{tab:system_par_beam}. 
\begin{table}[h!]
\begin{centering}
\caption{Notation used in the discretized beam example. \label{tab:system_par_beam}}
\begin{tabular}{|c|c|}
\hline 
Symbol & Meaning {(}unit{)}\tabularnewline
\hline 
\hline 
$L$ & Length of beam $(\unit[]{mm})$\tabularnewline
\hline 
$h$ & Height of beam $(\unit[]{mm})$\tabularnewline
\hline 
$b$ & Width of beam $(\unit[]{mm})$\tabularnewline
\hline 
$\rho$ & Density $(\unit[]{kg/mm^3})$\tabularnewline
\hline 
$E$ & Young's Modulus $(\unit[]{kPa})$\tabularnewline
\hline 
$I$ & Area moment of inertia $(\unit[]{mm^4})$\tabularnewline
\hline 
$\kappa$ & Coefficient cubic spring $(\unit[]{mN/mm^3})$  \tabularnewline
\hline 
$A$ & Cross-section of beam $(\unit[]{mm^2})$\tabularnewline
\hline 
$P$ & External forcing amplitude $(\unit[]{mN})$\tabularnewline
\hline
\end{tabular}
\par\end{centering}
\end{table}
The beam's neutral axis is the line of points coinciding with the $x$-axis. The Bernoulli hypothesis states that initially straight material lines, normal to the neutral axis, remain (a) straight and (b) inextensible, and (c) rotate as rigid lines to remain perpendicular to the beam's neutral axis after deformation. The transverse displacement of a material point with initial coordinates on the beam's neutral axis at $z=0$ is denoted by $w(x)$. The rotation angle of a transverse normal line about the $y$-axis is given by $-\partial_x w(x)$. We assume an isotropic, linearly elastic constitutive relation between the stresses and strains. This yields the following equations of motion
\begin{equation}
\rho A\frac{\partial^2 w(x,t)}{\partial{t^2}}-\rho I \frac{\partial^4 w(x,t)}{\partial x^2 \partial t^2} + EI \frac{\partial^4 w(x,t)}{\partial x^4} = 0. \label{eq:PDE_beam}
\end{equation}
We can neglect the mixed partial derivative term in Eq. (\ref{eq:PDE_beam}) by assuming that the thickness of the beam is small compared to its length, i.e., $h \ll L$ (see Reddy and Mahaffey \cite{Reddy2013}), we therefore can write Eq. (\ref{eq:PDE_beam}) as
\begin{equation}
\rho A\frac{\partial^2 w(x,t)}{\partial{t^2}} + EI \frac{\partial^4 w(x,t)}{\partial x^4} = 0. \label{eq:PDE_beam_2}
\end{equation}

We discretize Eq. (\ref{eq:PDE_beam_2}) and obtain a set of ordinary differential equations
\begin{equation}
\mtrx{M}\ddot{\mtrx{y}}+\mtrx{K}\mtrx{y}=\mtrx{0},
\end{equation} 
where $\mtrx{y}\in\mathbb{R}^{2m}=\mathbb{R}^n$, and $m$ is the number of elements used in the discretization. Each node of the beam has two coordinates related to the transverse displacement $w(x)$ and the rotation angle $-\partial_x w(x)$ of the cross section. Structural damping is assumed by considering the damping matrix
\begin{equation}
\mtrx{C} = \alpha \mtrx{M} + \beta \mtrx{K},
\end{equation}
with parameters $\alpha$ and $\beta$. We apply cosinusoidal external forcing on the transverse displacement coordinate at the free end of the beam with forcing frequency $\Omega$ and forcing amplitude $\varepsilon P$. Additionally, we add a cubic spring along this coordinate, with coefficient $\kappa$. As a result, the second-order equations of motion can be written as
\begin{equation}
\mtrx{M}\ddot{\mtrx{y}}+\mtrx{C}\dot{\mtrx{y}}+\mtrx{K}\mtrx{y}+\mtrx{g}(\mtrx{y},\dot{\mtrx{y}})=\varepsilon\mtrx{f}(\Omega t). \label{eq:eom_beam_2nd}
\end{equation} 
We give an illustration of the beam in Fig. \ref{fig:EB_beam}.
\begin{figure}[h!]
\centering
    \begin{subfigure}{1\textwidth}
        \centering
        \if\mycmd1
        \includegraphics[scale=1]{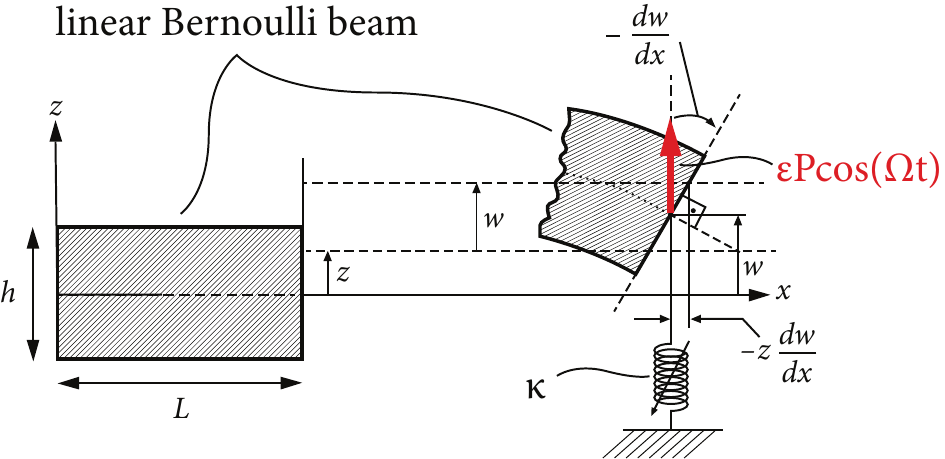}
        \fi
    \end{subfigure}
    \caption{Forced Bernoulli beam with a cubic spring. \label{fig:EB_beam}}
\end{figure}

We transform Eq. (\ref{eq:eom_beam_2nd}) to first-order form by setting $\mtrx{x}=[\mtrx{x}_1,\mtrx{x}_2]^\top=[\mtrx{y},\dot{\mtrx{y}}]^\top$ and apply a change of coordinates $\mtrx{x}=\mtrx{T}\mtrx{q}$, resulting in
\begin{align}
\dot{\mtrx{q}} & =
\mtrx{T}^{-1}\left(\begin{array}{cc}
\mtrx{0} & \mtrx{I}\\
-\mtrx{M}^{-1}\mtrx{K} & -\mtrx{M}^{-1}\mtrx{C}
\end{array}\right)\mtrx{T}\mtrx{q}
+\mtrx{T}^{-1}\left(\begin{array}{c}
\mtrx{0}\\
-\mtrx{M}^{-1}\mtrx{g}(\mtrx{T}\mtrx{q})
\end{array}\right)
+ \varepsilon\mtrx{T}^{-1}\left(
\begin{array}{c}
\mtrx{0}\\
\mtrx{M}^{-1}\mtrx{f}(\Omega t)
\end{array}
\right) \nonumber \\
&= \mtrx{\Lambda}\mtrx{q} +
\mtrx{T}^{-1}\left(\begin{array}{cc}
\mtrx{0} & \mtrx{0}\\
\mtrx{0} & \mtrx{M}^{-1}
\end{array}\right)
\left(\begin{array}{c}
0 \\
\vdots\\
-\kappa \left(\sum_{i=1}^{2n}[\mtrx{T}]_{n-1,i}q_i\right)^3 \\
0
\end{array}\right) +
 \varepsilon \mtrx{F}_\text{m}(\Omega t) \nonumber \\
&= \mtrx{\Lambda}\mtrx{q} + \mtrx{G}_\text{m}(\mtrx{q}) + \varepsilon \mtrx{F}_\text{m}(\Omega t) ,\label{eq:eom_beam_1st}
\end{align}
Using \ssm, we compute a third-order SSM reduced model of system (\ref{eq:eom_beam_1st}), which will take the following form
\begin{align}
\dot{\rho}&=a(\rho)+\varepsilon\left(f_1(\rho,\Omega)\cos(\psi)+f_2(\rho,\Omega)\sin(\psi)\right),\label{eq:beam_red_1} \\
\dot{\psi}&=(b(\rho)-\Omega)+\frac{\varepsilon}{\rho}\left(g_1(\rho,\Omega)\cos(\psi)-g_2(\rho,\Omega)\sin(\psi)\right),\label{eq:beam_red_2}
\end{align}
where
\begin{align}
a(\rho)   &= \text{Re}(\lambda_1)\rho+\text{Re}(\gamma_1)\rho^{3}, \nonumber\\
b(\rho)   &= \text{Im}(\lambda_1)+\text{Im}(\gamma_1)\rho^{2},  \nonumber\\
f_1(\rho,\Omega) &= \text{Re}(c_{1,(0,0)}) + \left(\text{Re}(c_{1,(1,1)}(\Omega))+\text{Re}(d_{1,(2,0)}(\Omega))\right)\rho^{2}, \nonumber\\
f_2(\rho,\Omega) &= \text{Im}(c_{1,(0,0)}) + \left(\text{Im}(c_{1,(1,1)}(\Omega))-\text{Im}(d_{1,(2,0)}(\Omega))\right)\rho^{2},\nonumber \\
g_1(\rho,\Omega) &= \text{Im}(c_{1,(0,0)}) + \left(\text{Im}(c_{1,(1,1)}(\Omega))+\text{Im}(d_{1,(2,0)}(\Omega))\right)\rho^{2},\nonumber \\
g_2(\rho,\Omega) &= \text{Re}(c_{1,(0,0)}) + \left(\text{Re}(c_{1,(1,1)}(\Omega))-\text{Re}(d_{1,(2,0)}(\Omega))\right)\rho^{2}.  \nonumber
\end{align}
We can explicitly compute the autonomous and non-autonomous SSM coefficients, which are used to verify the output given by \ssm,
\begin{align}
\gamma_1 &= -3 \kappa [\tilde{\mtrx{B}}]_{1,2n-1}[\mtrx{T}]_{n-1,1}^2[\mtrx{T}]_{n-1,2}, \\
c_{1,(0,0)} &= \frac{[\tilde{\mtrx{B}}]_{1,2n-1}P}{2}, \\
c_{1,(1,1)} &= 6\kappa [\tilde{\mtrx{B}}]_{1,2n-1}[\mtrx{T}]_{n-1,1}[\mtrx{T}]_{n-1,2}\sum_{j=2}^{2n}\frac{[\mtrx{T}]_{n-1,j}[\tilde{\mtrx{B}}]_{j,2n-1}P}{2(\lambda_j-\mi \Omega)}, \\
d_{1,(2,0)} &= 3\kappa [\tilde{\mtrx{B}}]_{1,2n-1}[\mtrx{T}]_{n-1,1}^2\sum_{\substack{j=1\\ j \neq 2}}^{2n}\frac{[\mtrx{T}]_{n-1,j}[\tilde{\mtrx{B}}]_{j,2n-1}P}{2(\lambda_j+\mi \Omega)},
\end{align}
where the matrix $\tilde{\mtrx{B}}$ is defined as
\begin{equation}
\tilde{\mtrx{B}}=
\mtrx{T}^{-1}\left(\begin{array}{cc}
\mtrx{0} & \mtrx{0}\\
\mtrx{0} & \mtrx{M}^{-1}
\end{array}\right).
\end{equation}

\subsection{Numerical results}
In our upcoming comparison, the collocation computations were performed on a remote Intel Xeon E5-2680v3 processor (3.3 GHz) on the ETH cluster due to large computational times. The SSM and HB computations were performed on an Intel Xeon X5675 processor (3.07 GHz) on a local workstation.

We now compute the forced-response curves around the first vibration mode of the discretized beam model described above. The FRCs will be obtained independently from SSM theory, the harmonic balance method and a collocation method. We list the chosen geometric and material parameter values in Table \ref{tab:system_par_beam_ex_1}. 
\begin{table}[H]
\begin{centering}
\caption{Geometric and material parameters for the Bernoulli beam. \label{tab:system_par_beam_ex_1}}
\begin{tabular}{|c|c|}
\hline 
Symbol & Value \tabularnewline
\hline 
\hline 
$L$ & $\unit[2700]{mm}$\tabularnewline
\hline 
$h$ & $\unit[10]{mm}$ \tabularnewline
\hline 
$b$ & $\unit[10]{mm}$ \tabularnewline
\hline 
$\rho$ & $\unit[1780\cdot 10^{-9}]{kg/mm^3}$ \tabularnewline
\hline 
$E$ & $\unit[45\cdot 10^6]{kPa}$ \tabularnewline
\hline 
$\kappa$ & $\unit[4]{mN/mm^3}$ \tabularnewline
\hline 
$\alpha$ & $\unit[1.25\cdot 10^{-4}]{s^{-1}}$ \tabularnewline
\hline 
$\beta$ & $\unit[2.5 \cdot 10^{-4}]{s}$ \tabularnewline
\hline
$P$ & $\unit[0.1]{mN}$ \tabularnewline
\hline
\end{tabular}
\par\end{centering}
\end{table}
As system (\ref{eq:eom_beam_1st}) is a discretized version of Eq. (\ref{eq:PDE_beam_2}), the first natural frequency of the conservative, unforced, fixed-free beam, consisting of $m$ elements, will approximate 
\begin{equation}
\omega_1 = (\beta_1 l)^2 \sqrt{\frac{EI}{\rho A l^4}}\approx \unit[7]{rad/s}, \quad \beta_1 l = 1.875104, \label{eq:cont_1st_eig}
\end{equation}
for an increasing value of $m$ (see Rao \cite{singiresu1995mechanical}). If the damping is small, the imaginary part of $\lambda_1$ will approximately be equal to $\omega_1$ (cf. G{\'e}radin and Rixen \cite{Geradin2014}). 

We used the \texttt{ode{\_}isol2po} toolbox constructor in \textsc{coco} \cite{Dankowicz2013} for continuation along a family of single-segment periodic orbits from an initial solution guess. The single-segment collocation zero problem is initially constructed on a default mesh with 10 intervals, 5 base points and  4 collocation nodes in each interval. The continuation algorithm is then instructed to make adaptive changes to the problem discretization after each step of continuation.  

We also used the \textsc{nl}vib tool \cite{Malte2018}, which implements the HB method coupled to a path-continuation procedure. In the HB method, it is assumed that the system has a steady-state solution represented by a Fourier series
\begin{equation}
\mtrx{y}=\text{Re}\left(\sum_{k=0}^\infty \mtrx{c}_k \me^{\mi k \Omega t}\right), \label{eq:fourier_assump}
\end{equation}
where $\mtrx{c}_k\in\mathbb{C}^n$ is a vector containing the complex Fourier coefficients corresponding to the $k^\text{th}$ harmonic. Furthermore, it is assumed that the nonlinear force vector $\mtrx{g}(\mtrx{y},\dot{\mtrx{y}})$ can be approximated by a Fourier series as well. 

By substituting the assumed solution (\ref{eq:fourier_assump}) into the original ordinary differential equations (\ref{eq:eom_beam_2nd}) and restricting the result to finitely many harmonics $H$ (we will use $H=10$), the original equations are transformed into a set of nonlinear algebraic equations
\begin{equation}
\left(-(k\Omega)^2\mtrx{M}+\mi k\Omega\mtrx{C} + \mtrx{K} \right)\mtrx{c}_k + \mtrx{f}_{\text{nl},k}(\mtrx{c}_0,\ldots,\mtrx{c}_H)-\mtrx{f}_{\text{ext},k}=\mtrx{0},\quad k=0,\ldots,H \label{eq:HBmethod}
\end{equation}
to be solved simultaneously for all $\mtrx{c}_k$, with $k=0,\ldots,H$. This is typically done using a Newton-Raphson iteration scheme.  

To evaluate the nonlinear force vector $\mtrx{f}_{\text{nl},k}(\mtrx{c}_0,\ldots,\mtrx{c}_H)$ in (\ref{eq:HBmethod}), \textsc{nl}vib tool uses the Alternating-Frequency-Time (AFT) method, proposed first by Cameron et al. \cite{cameron1989alternating}. This algorithm uses the inverse Fourier transform of the positions and velocities in the frequency domain, creating a sampled time signal over one period of oscillation. The time signal is then substituted into the nonlinear force vector $\mtrx{g}(\mtrx{y},\dot{\mtrx{y}})$ and the resulting output signal is in turn transformed back to the frequency domain using a Fourier transformation. For several implementations of the AFT method we refer to\cite{detroux2015harmonic,cardona1998fast,zhang2014harmonic,sinou2007non}. 

A shortcoming of the HB method, as compared to SSM theory and the collocation method used by \textsc{coco}, is that it does not provide any information about the stability of the solutions, which has to be analyzed in a separate effort. As described in Detroux et al. \cite{detroux2015harmonic}, a variant of Floquet theory can be used in order to identify the stability of the solutions, which is applicable in the frequency domain and is known as Hill's method \cite{hill1886part}. This separate analysis has not been implemented in the current work. 

We now compute the forced-response curve, around $\omega_1$ (\ref{eq:cont_1st_eig}), over the interval $S_\Omega=[6.88,7.12]$ for an increasing number of elements $m$ and $\varepsilon=0.002$. We verify our results and compare the recorded computational times using \ssm with the numerical continuation package \coco and the harmonic balance method. The corresponding computational times are listed in Fig. \ref{fig:bar_plot}.
\begin{figure}
\centering
     \begin{subfigure}{1\textwidth}
        \centering
        \if\mycmd1
        \includegraphics[scale=0.19]{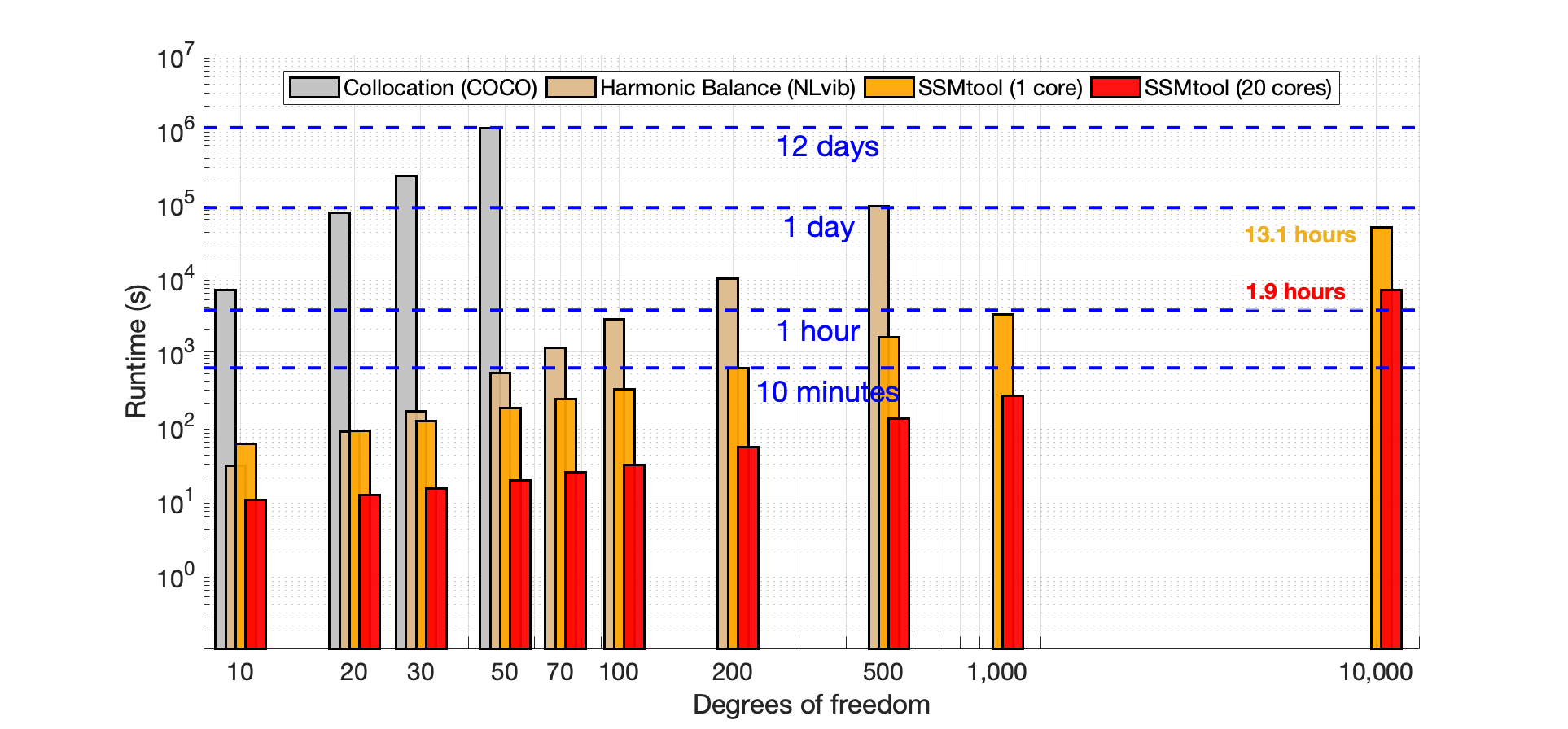}
        \fi
        \vspace{3mm}
    \end{subfigure}
    \caption{Computational times to extract the forced-response curve around the first vibration mode of a cantilevered Bernoulli beam with a cubic spring over the interval $S_\Omega=[6.88,7.12]$, using collocation, harmonic balance and \ssm. \label{fig:bar_plot}}
\end{figure}

As can be seen in Fig. \ref{fig:bar_plot}, the collocation based method with \textsc{coco} takes 12 full days to compute the forced-response curve, over the interval $S_\Omega$, for a 50-degrees-of-freedom system and due to this reason has not been used for higher-degrees-of-freedom simulations. For the discretized beam with 500-degrees-of-freedom, the HB method with 10 harmonics takes around 1 day to compute the forced response curve, where the number of nonlinear algebraic equations and unknowns is given by
\begin{equation}
p = n(2H+1). \nonumber 
\end{equation}
For the 1000 degrees-of-freedom system, the total number of nonlinear algebraic equations is $p=21000$, which has to be solved for the $21000$ unknown Fourier coefficients. This becomes unfeasible using the available \matlab implementation of the HB method. 

For the \ssm calculation, the 10,000 degrees of freedom example takes a total of 13 hours when computed on a single core. Here we sampled the frequency interval $S_\Omega$ for 60 frequency values $\Omega_i$ and computed the third-order approximation for the non-autonomous SSM. As the autonomous part does not depend on the forcing frequency $\Omega$, we only have to compute this part once. The non-autonomous part is recalculated for different samples $\Omega_i$, which makes it possible to parallelize the non-autonomous computations by dividing the frequency samples over different cores. Running the non-autonomous part of the SSM computation on 20 cores reduces the total computational time from 13 hours to 2 hours. 

The resulting FRCs corresponding to the absolute maximum displacement during one period of oscillation of the transverse component at the free end of the beam, for $n=\{10,50,500, 10000\}$ over the interval $S_\Omega$, are listed in Figs. \ref{fig:FRC_comparison}. In Fig. \ref{fig:phase_plane} we illustrate the phase plane of the two-dimensional SSM-reduced system extracted from the 100 degrees-of-freedom beam example, showing how the domain of attraction of the higher amplitude stable fixed point reduces up to the point where a saddle-node bifurcation occurs, which is where the stable and saddle-type fixed points collide and annihilate each other. 

\begin{figure}
\centering
     \begin{subfigure}{0.45\textwidth}
        \centering
        \if\mycmd1
        \includegraphics[scale=0.22]{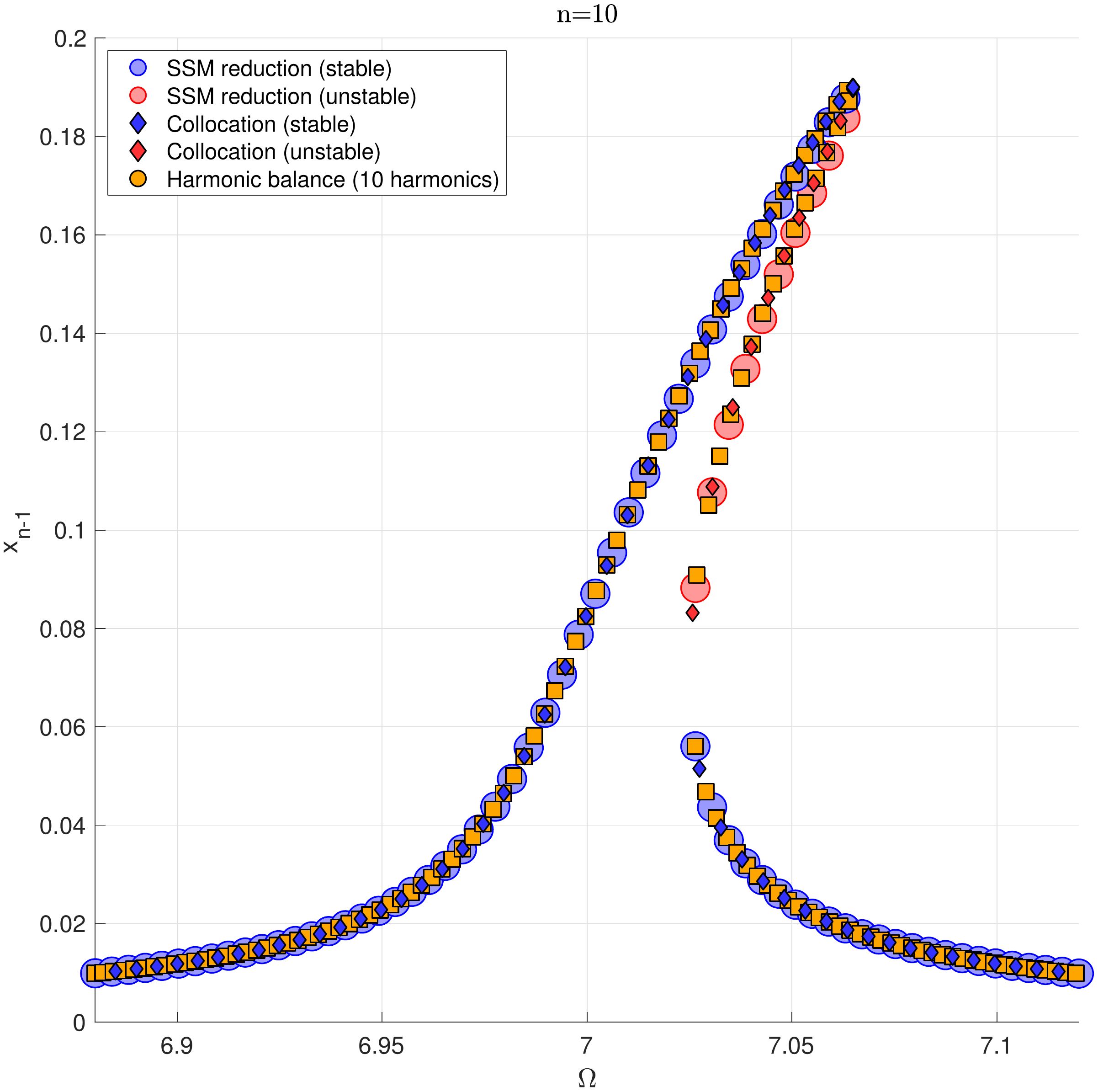}
        \fi
        \caption{}
        \vspace{3mm}
    \end{subfigure}
    \begin{subfigure}{0.45\textwidth}
        \centering
        \if\mycmd1
        \includegraphics[scale=0.22]{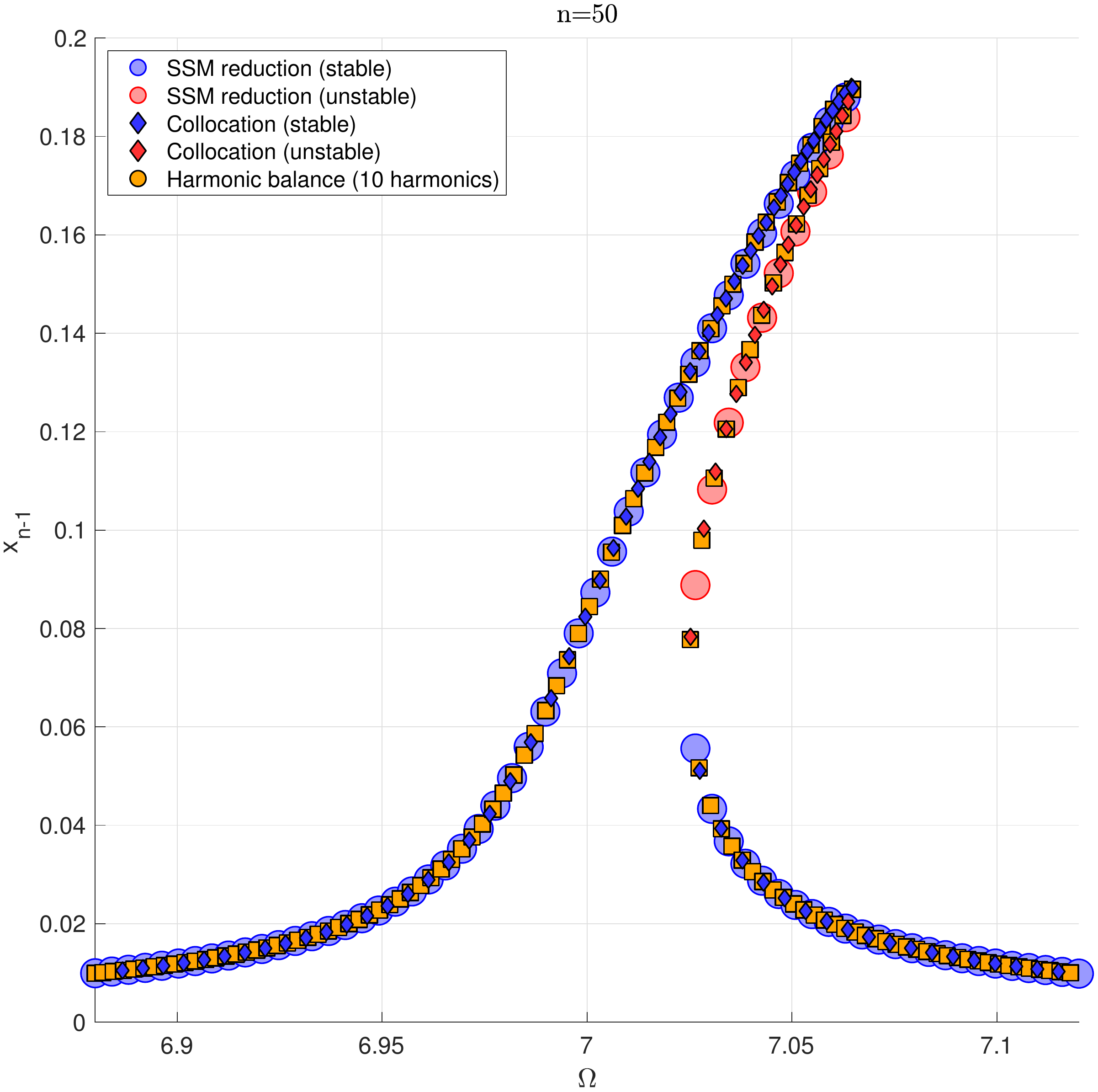}
        \fi
        \caption{}
        \vspace{3mm}
    \end{subfigure}

     \begin{subfigure}{0.45\textwidth}
        \centering
        \if\mycmd1
        \includegraphics[scale=0.22]{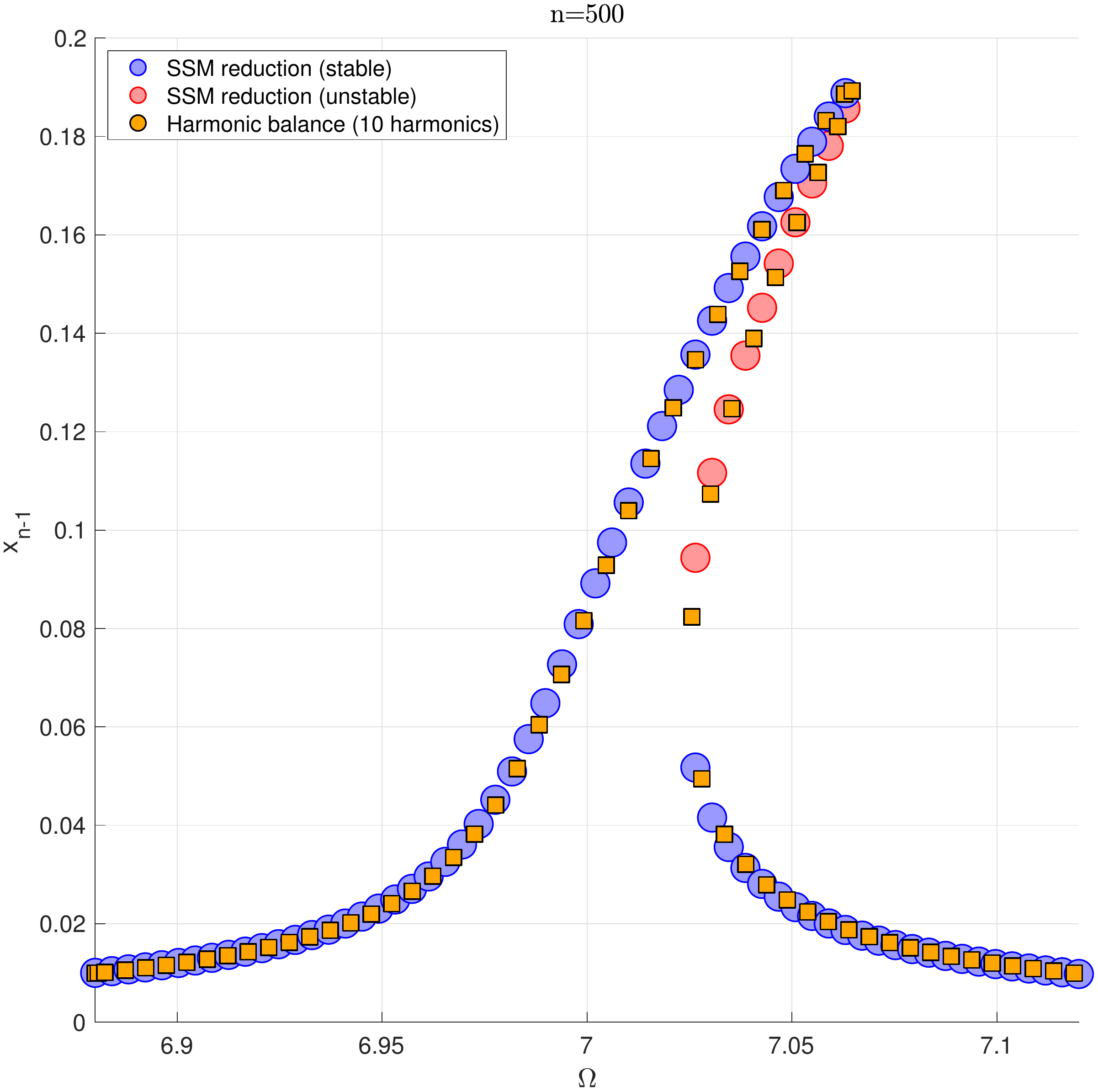}
        \fi
        \caption{}
        \vspace{3mm}
    \end{subfigure}
    \begin{subfigure}{0.45\textwidth}
        \centering
        \if\mycmd1
        \includegraphics[scale=0.22]{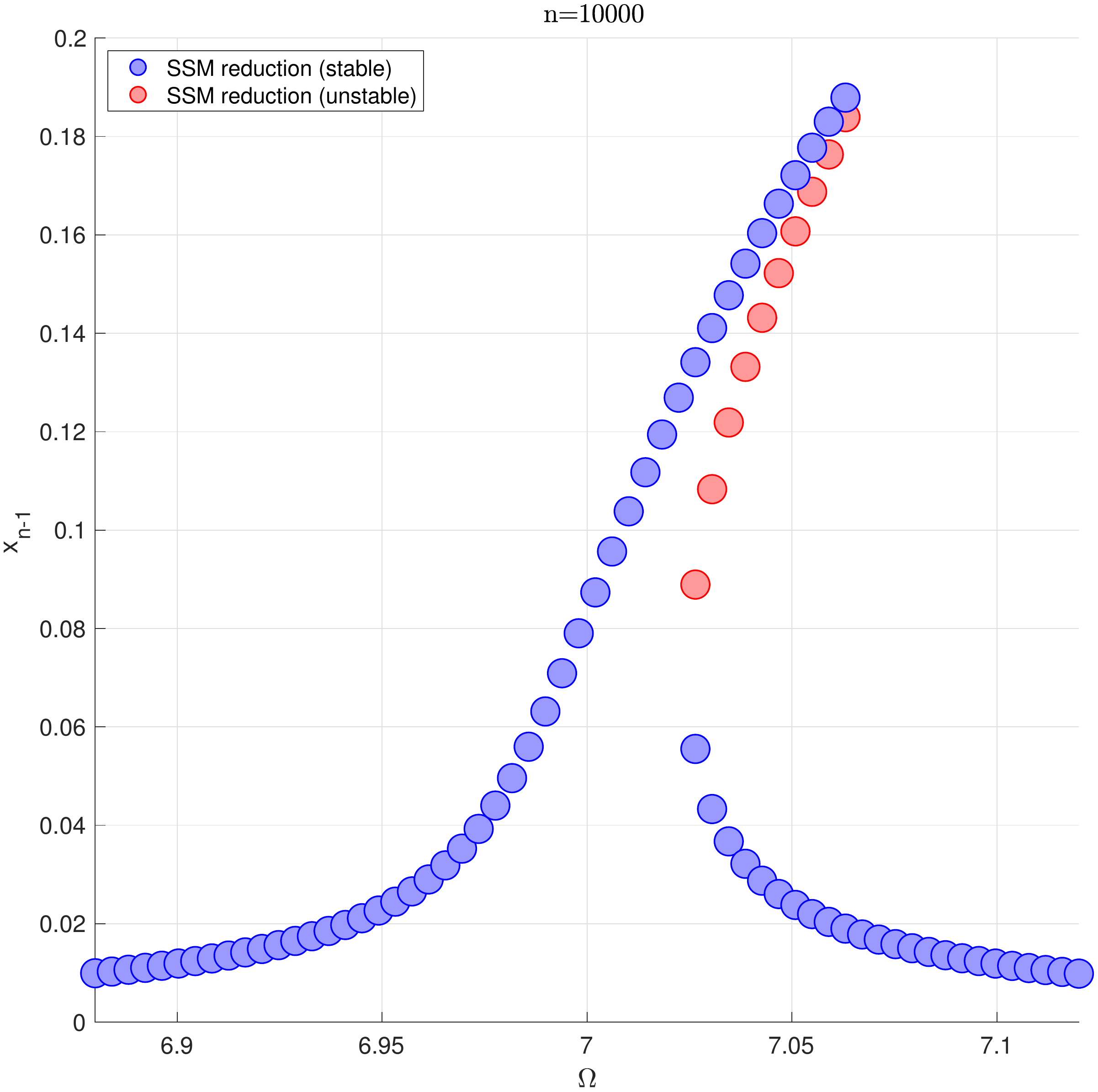}
        \fi
        \caption{}
        \vspace{3mm}
    \end{subfigure}
    \caption{Extracted forced response curves for $x_{n-1}$, using a third-order SSM reduced model, collocation and the harmonic balance method, for an increasing number of degrees of freedom $n$, where $n=\{10,50,500,10000\}$ in Figs (a), (b), (c) and (d), respectively.\label{fig:FRC_comparison}}
\end{figure}


\begin{figure}
\centering
     \begin{subfigure}{0.45\textwidth}
        \centering
        \if\mycmd1
        \includegraphics[scale=0.25]{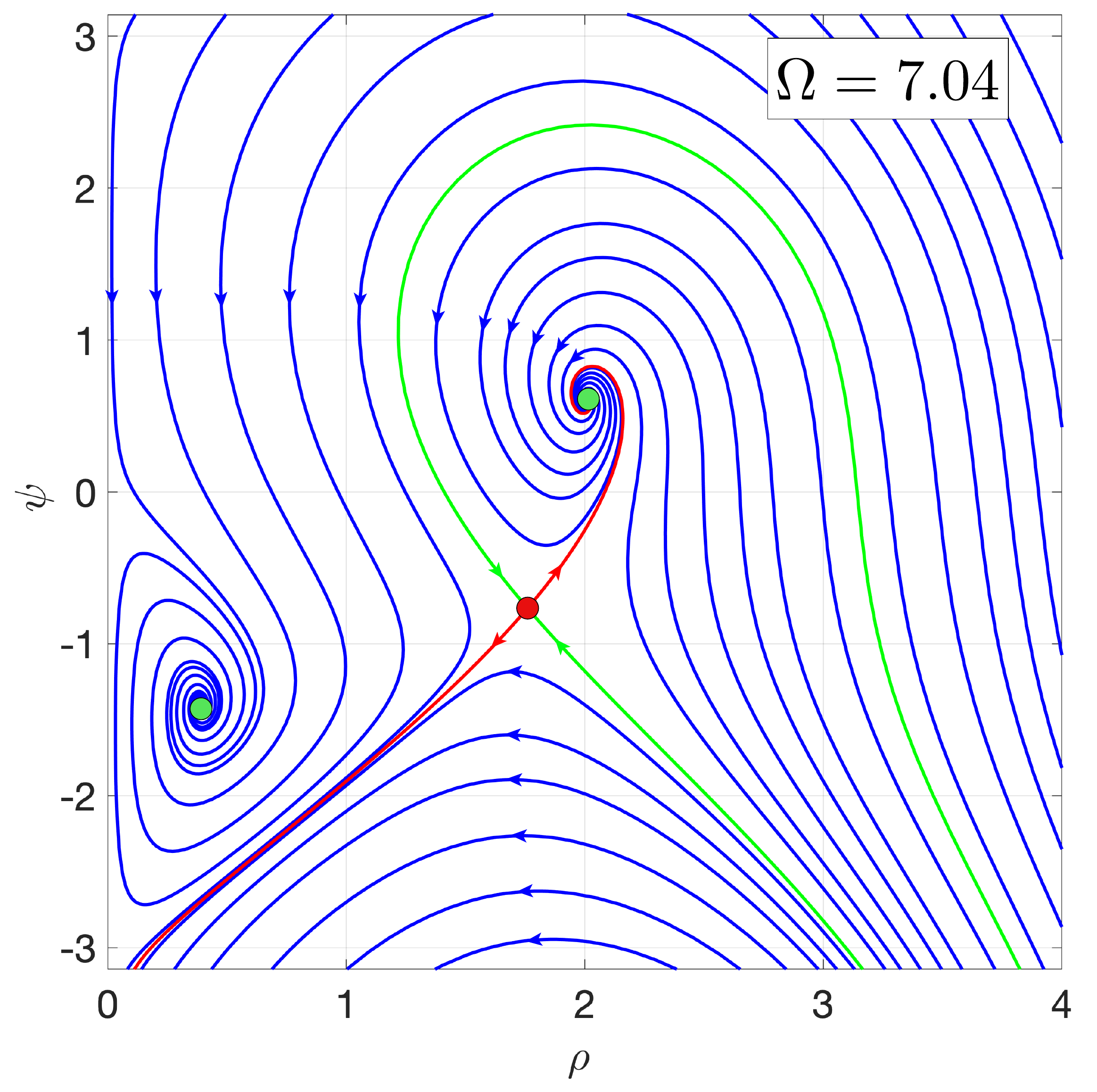}
        \fi
        \caption{}
        \vspace{3mm}
    \end{subfigure}
    \begin{subfigure}{0.45\textwidth}
        \centering
        \if\mycmd1
        \includegraphics[scale=0.25]{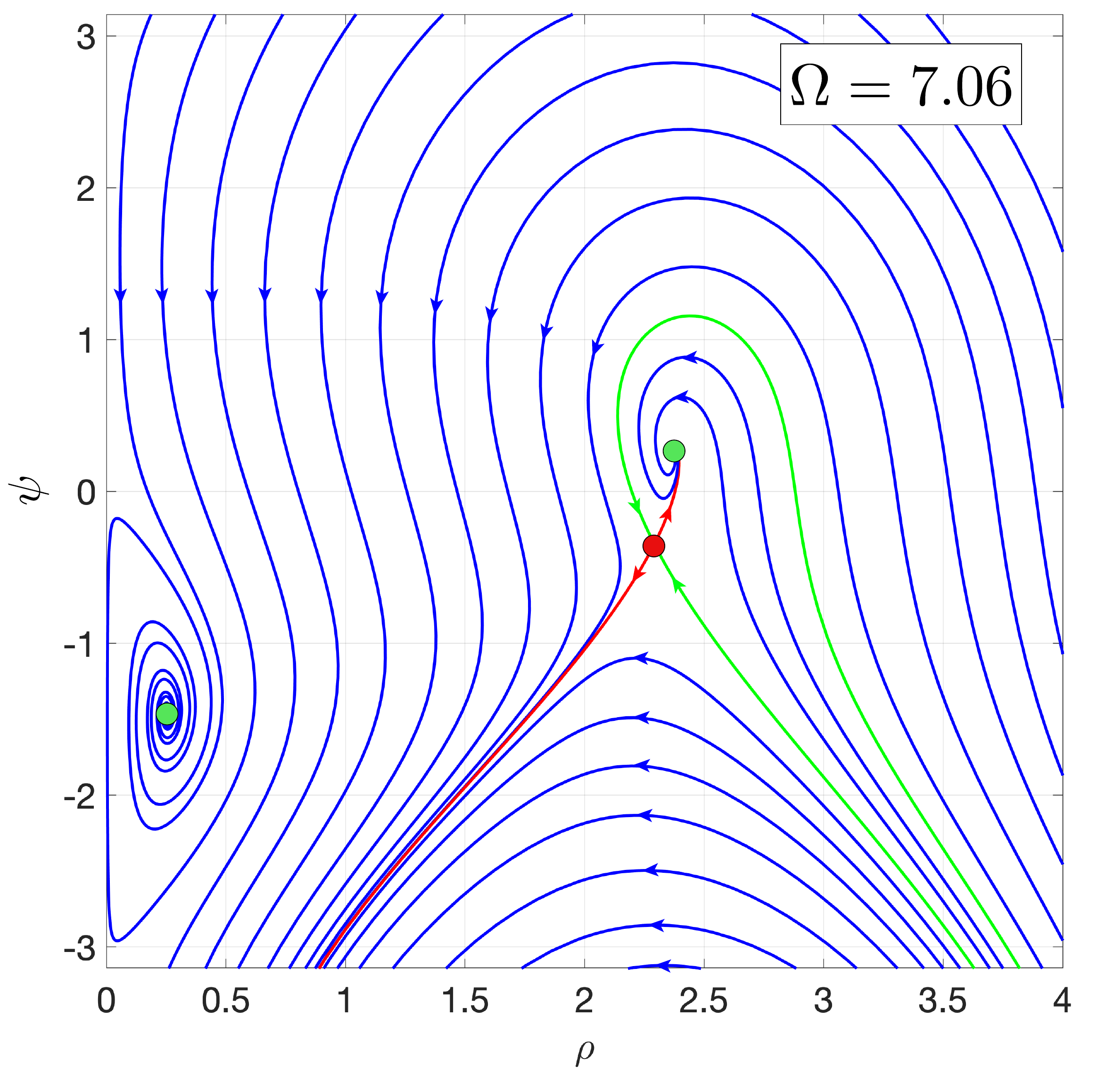}
        \fi
        \caption{}
        \vspace{3mm}
    \end{subfigure}

     \begin{subfigure}{0.45\textwidth}
        \centering
        \if\mycmd1
        \includegraphics[scale=0.25]{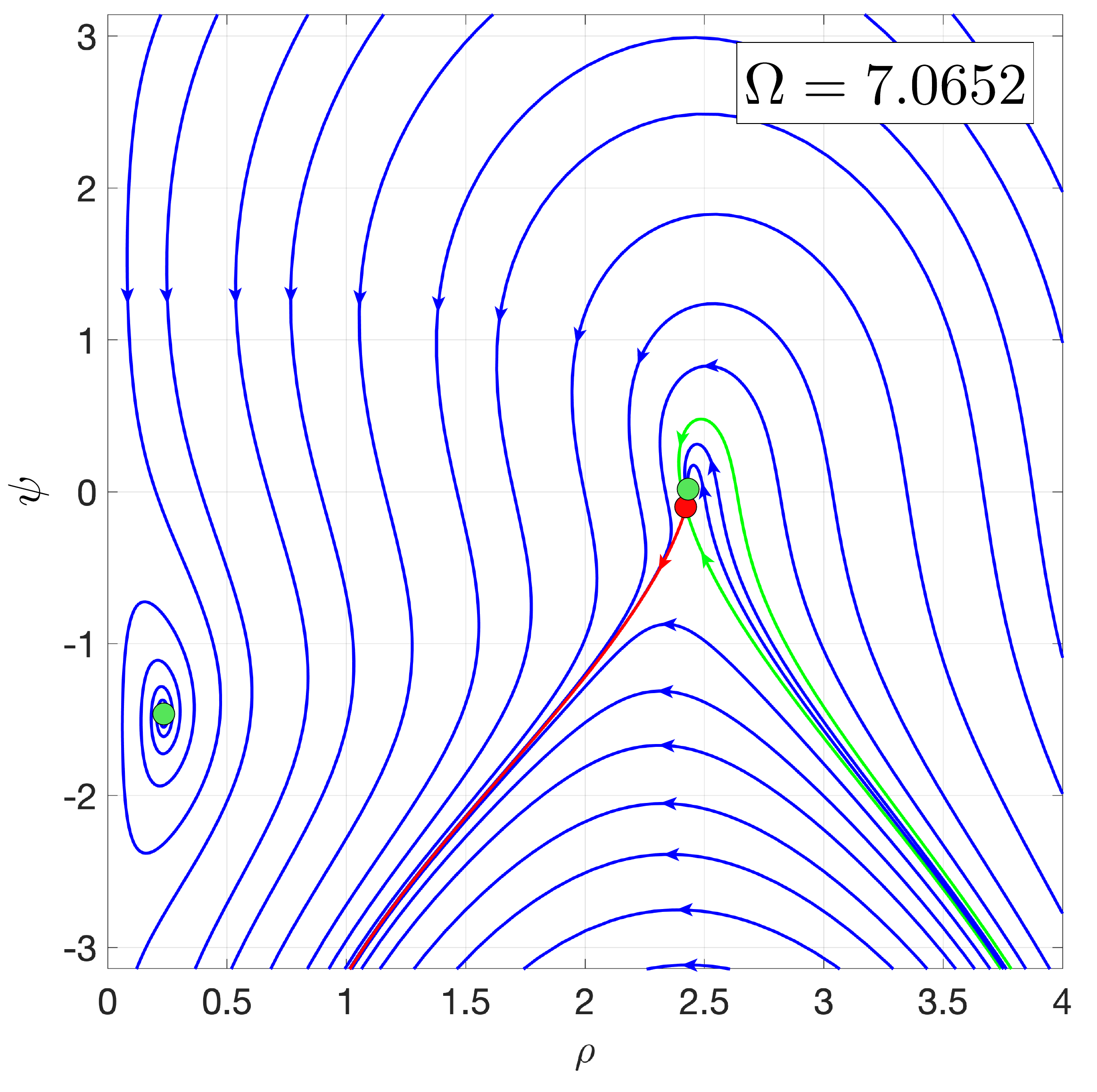}
        \fi
        \caption{}
        \vspace{3mm}
    \end{subfigure}
    \begin{subfigure}{0.45\textwidth}
        \centering
        \if\mycmd1
        \includegraphics[scale=0.25]{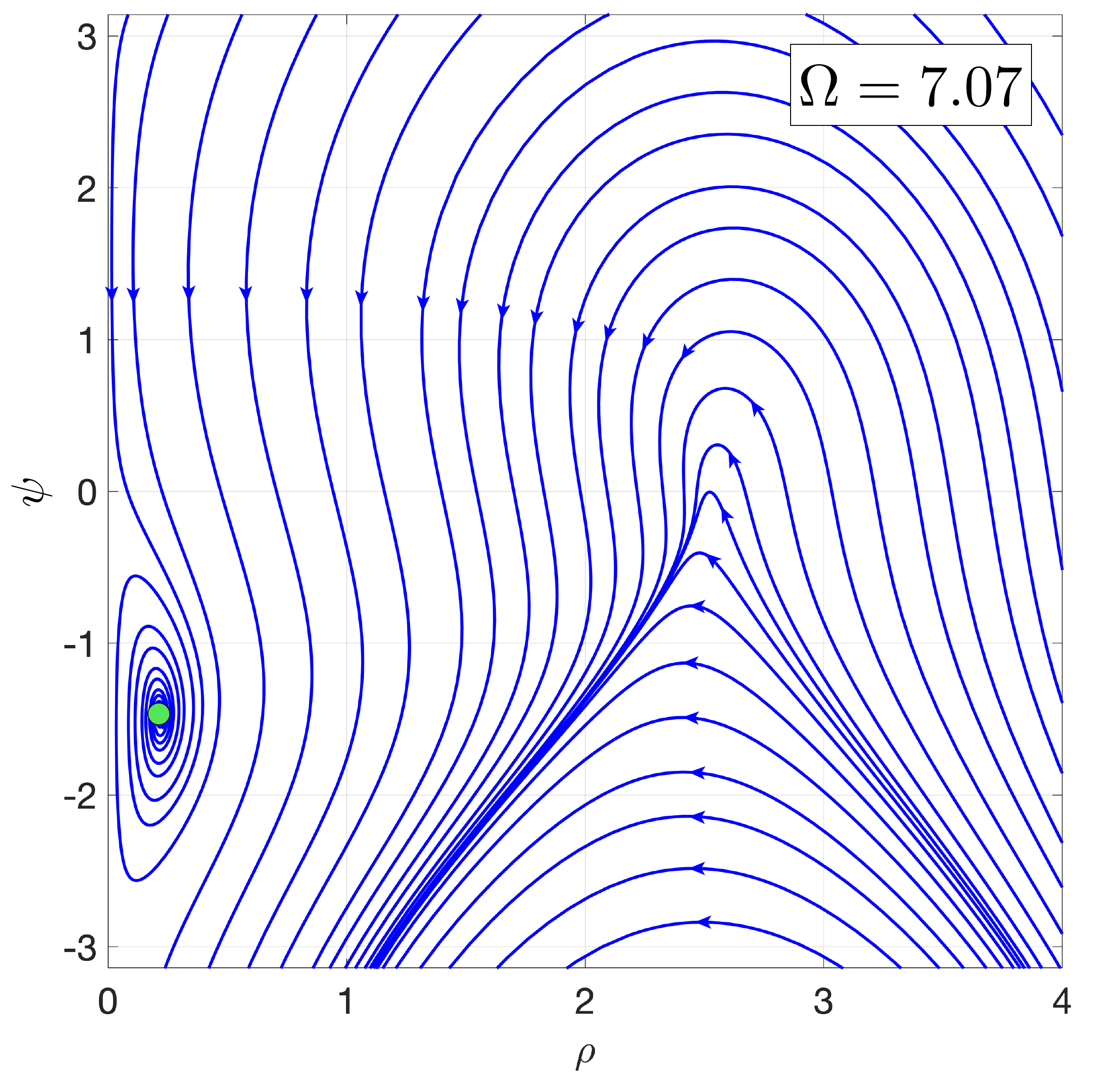}
        \fi
        \caption{}
        \vspace{3mm}
    \end{subfigure}
    \caption{Phase plane of the two-dimensional SSM-reduced system extracted from the 100 degrees-of-freedom beam example for different forcing frequencies $\Omega$ and fixed forcing amplitude $\varepsilon=0.002$. The Figures (a), (b) and (c), the reduced system has a total of three fixed points, of which two are stable spirals and one is a saddle. As the forcing frequency is increased (cf. Fig. (d)),  a saddle-node bifurcation occurs where the two higher-amplitude fixed points collide and annihilate each other. The stable and unstable manifolds of the saddle-type fixed point are shown in green and red. Notice how the domain of attraction of the higher amplitude stable fixed point reduces significantly in area as the forcing frequency is increased, making it harder to end up in this particular fixed point. \label{fig:phase_plane}}
\end{figure}

\section{Conclusion}
In this work, we have used the reduced dynamics on two-dimensional time-periodic spectral submanifolds (SSMs) to extract forced-response curves (FRCs) around the vibration modes of nonlinear non-conservative mechanical systems. We compared the computational times needed to extract such FRCs from systems with an increasing number of degrees of freedom, using SSM theory, the harmonic balance (HB) method and a collocation method implemented in the \texttt{po} toolbox of \coco. 

Varying the number of degrees of freedom, from 10 to a 10,000, we have found that extracting the FRC using the HB method and the collocation method becomes rapidly intractable. However, using \ssm, a 10,000-degree-of-freedom system takes approximately 13 hours to obtain the FRC over a predefined set of frequency values. 

An additional advantage of the present approach is that SSM computations can be parallelized. The frequency domain of interest can be divided into subsets and each computation over such a subset can be sent to a different core. For the 10,000 degrees-of-freedom system, running the \ssm computation in parallel on 20 cores reduces the computational time from 13 hours to approximately 2 hours. These speeds and corresponding degrees of freedom appear certainly out of reach for any other approach that we are aware of for steady-state calculations in periodically forced nonlinear mechanical systems. 

We have visualized the phase space of the two-dimensional SSM-reduced systems. Doing so we have reproduced the behavior commonly observed in experiments: during a frequency sweep of the system, following the higher-amplitude stable periodic solution branch becomes harder near folding points. Indeed, as our analysis reveals, small perturbations can cause the response of the system to escape the domain of attraction of the higher-amplitude stable periodic orbit, ending up in the domain of attraction of the lower-amplitude stable periodic solution. Specifically, the domain of attraction of the higher-amplitude fixed point, for the SSM-reduced system, shrinks in area up to the point where it completely vanishes during a saddle-node bifurcation. 

When the forcing frequency, $\Omega$, and the forcing amplitude, $\varepsilon$, are fixed, we showed that the zeros of the reduced dynamics lie on an ellipse-shaped curve, which gives a new geometric interpretation of the family of periodic orbits of the full system. Additionally, if we reduced our analysis to the setting of Breunung and Haller \cite{Breunung2017} and computed the non-autonomous part of the SSM only up to zeroth order in the parameterization coordinates, the ellipse would reduce to a circle. 

In summary, we find that spectral submanifolds provide a mathematically exact model reduction tool for high-degree-of-freedom nonlinear mechanical systems at previously unthinkable speeds. The reduction method does not require the numerical solution of differential equations: all effort goes into constructing appropriate matrices corresponding to a linear system of equations from which the solution describes the SSM and its reduced dynamics. Locating steady states then requires solving a two-dimensional algebraic system of equations, which is practically instantaneous. 

The main performance limitation for SSM-based model reduction is not processor speed but memory needs, which depends on the structure of the nonlinearities of the mechanical system. On the positive side, the storage requirements for SSM coefficients can be significantly optimized relative to the proof-of-concept approach presented here. This optimization is an improvement of \ssm and is currently ongoing work that will be published in the future. 

\section{Compliance with Ethical Standards}
Conflict of interest: The authors declare that they have no conflict of interest.

\appendix
\section{Proof of Theorem \ref{thrm:coef_eq_auto}\label{app:coef_eq_auto}}
For row $i$, the $\mtrx{k}^\text{th}$-power terms on the right-hand side of Eq. (\ref{eq:auto_SSM}) can be expressed as
\begin{equation}
\left[D_\mtrx{s} \mtrx{W}_0(\mtrx{s})\mtrx{R}_0(\mtrx{s})\right]^{\mtrx{k}}_i = \sum_{j=1}^{2}\sum_{\substack{\mtrx{m}\leq\tilde{\mtrx{k}}_j \\ m_j>0}} m_j W_{i,\mtrx{m}}^0 R_{j,\tilde{\mtrx{k}}_j-\mtrx{m}}^0
\end{equation}
The $\mtrx{k}^\text{th}$-power terms on the left-hand side of the $i^\text{th}$ row of Eq. (\ref{eq:auto_SSM}) can be written as
\begin{align}
\left[\gmtrx{\Lambda}\mtrx{W}_0(\mtrx{s})\right]^\mtrx{k}_i &= \lambda_iW_{i,\mtrx{k}}^0, \\
\left[\mtrx{G}_\text{m}(\mtrx{W}_0(\mtrx{s}))\right]^{\mtrx{k}}_i &=\left[g_i(\mtrx{W}_0(\mtrx{s}))\right]_\mtrx{k}.
\end{align}
where we have made use of the multi-index notation
\begin{equation}
\mtrx{m}\in\mathbb{N}_0^2,\quad \mtrx{k}\in\mathbb{N}_0^2,\quad \tilde{\mtrx{k}}_j = \mtrx{k}+\mtrx{e}_j,
\end{equation}
with $\mtrx{e}_j$ denoting a unit vector. 

The coefficient equation related to the $\textbf{k}^\text{th}$-power term of the $i^\text{th}$ row of the autonomous invariance Eq. (\ref{eq:auto_SSM}) can now be rewritten as
\begin{equation}
\left(\lambda_i-\sum_{j=1}^{2}k_j\lambda_j\right)W_{i,\mtrx{k}}^0 = \sum_{j=1}^{2}\delta_{ij}R_{j,\mtrx{k}}^0 + Q_{i,\mtrx{k}},
\end{equation}
where $Q_{i,\mtrx{k}}$ is defined as
\begin{align}
Q_{i,\mtrx{k}} = \sum_{j=1}^{2}\sum_{\substack{\mtrx{m}\leq\tilde{\mtrx{k}}_j \\ \mtrx{m}\neq \mtrx{e}_j\\ \mtrx{m}\neq\mtrx{k}_{\text{ }} \\ m_j>0}} m_j W_{i,\mtrx{m}}^0 R_{j,\tilde{\mtrx{k}}_j-\mtrx{m}}^0-\left[g_i(\mtrx{W}_0(\mtrx{s}))\right]_\mtrx{k}, \nonumber
\end{align}
which proves the result stated in Theorem \ref{thrm:coef_eq_auto}. \qed

\section{Proof of Theorem \ref{thrm:coef_eq}\label{app:coef_eq}}
Assuming that $\phi\in S^1$, we obtain that for the $i^\text{th}$ row, the $\mtrx{k}^\text{th}$-power terms on the right-hand side of Eq. (\ref{eq:invar_O_eps}) can be expressed as
\begin{align}
\left[D_\mtrx{s}\mtrx{W}_0(\mtrx{s})\mtrx{R}_1(\mtrx{s},\phi)\right]^{\mtrx{k}}_i &= \sum_{j=1}^{2}\sum_{\substack{\mtrx{m}\leq\tilde{\mtrx{k}}_j \\ m_j>0}} m_j W_{i,\mtrx{m}}^0 R_{j,\tilde{\mtrx{k}}_j-\mtrx{m}}^1(\phi), \\
\left[D_\mtrx{s}\mtrx{W}_1(\mtrx{s},\phi)\mtrx{R}_0(\mtrx{s})\right]^{\mtrx{k}}_i &= \sum_{j=1}^{2}\sum_{\substack{\mtrx{m}\leq\tilde{\mtrx{k}}_j \\ m_j>0}} m_j W_{i,\mtrx{m}}^1(\phi)R_{j,\tilde{\mtrx{k}}_j-\mtrx{m}}^0, \\
\left[D_{\phi} \mtrx{W}_1(\mtrx{s},\phi)\Omega\right]^{\mtrx{k}}_i &=D_{\phi}W_{i,\mtrx{k}}^1(\phi)\Omega .
\end{align}
The $\mtrx{k}^\text{th}$-power terms on the left-hand side of the $i^\text{th}$ row of Eq. (\ref{eq:invar_O_eps}) can be written as
\begin{align}
\left[\gmtrx{\Lambda}\mtrx{W}_1(\mtrx{s},\phi)\right]^\mtrx{k}_i &= \lambda_iW_{i,\mtrx{k}}^1(\phi), \\
\left[D_\mtrx{q}\mtrx{G}_\text{m}(\mtrx{W}_0(\mtrx{s}))\mtrx{W}_1(\mtrx{s},\phi)\right]^{\mtrx{k}}_i &=\left[\sum_{j=1}^{2n}D_{q_j}g_i(\mtrx{W}_0(\mtrx{s}))w_j^1(\mtrx{s},\phi)\right]_\mtrx{k}, \\
\left[\mtrx{F}_\text{m}(\phi)\right]^\mtrx{k}_i&=F_{i,\mtrx{k}}(\phi). 
\end{align}
Therefore, the coefficient equation related to the $\textbf{k}^\text{th}$-power term of the $i^\text{th}$ row of the non-autonomous invariance Eq. (\ref{eq:invar_O_eps}) is
\begin{equation}
\left(\lambda_i-\sum_{j=1}^{2}k_j\lambda_j\right)W_{i,\mtrx{k}}^1(\phi)-D_{\phi} W_{i,\mtrx{k}}^1(\phi)\Omega = \sum_{j=1}^{2}\delta_{ij}R_{j,\mtrx{k}}^1(\phi) + P_{i,\mtrx{k}}(\phi),
\end{equation}
where 
\begin{align}
P_{i,\mtrx{k}}(\phi) &= \sum_{j=1}^{2}\sum_{\substack{\mtrx{m}\leq\tilde{\mtrx{k}}_j \\ \mtrx{m}\neq \mtrx{e}_j \\ m_j>0}} m_j W_{i,\mtrx{m}}^0 R_{j,\tilde{\mtrx{k}}_j-\mtrx{m}}^1(\phi)+\sum_{j=1}^{2}\sum_{\substack{\mtrx{m}\leq\tilde{\mtrx{k}}_j \\ \mtrx{m}\neq\mtrx{k} \\ m_j>0}} m_j W_{i,\mtrx{m}}^1(\phi)R_{j,\tilde{\mtrx{k}}_j-\mtrx{m}}^0 \\
&-F_{i,\mtrx{k}}(\phi) -\left[\sum_{j=1}^{2n}D_{q_j}g_i(\mtrx{W}_0(\mtrx{s}))w_j^1(\mtrx{s},\phi)\right]_\mtrx{k}, \nonumber
\end{align}
which concludes the proof of Theorem \ref{thrm:coef_eq}.\qed
\section{Proof of Theorem \ref{thm:red_dyn}\label{app:red_dyn}}
The $\mathcal{O}(\varepsilon)$ approximation of the reduced dynamics for $\mtrx{s}$ can be written as
\begin{equation}
\dot{\mtrx{s}} = \mtrx{R}(\mtrx{s},\phi)=\mtrx{R}_0(\mtrx{s})+\varepsilon\mtrx{R}_1(\mtrx{s},\phi),\label{eq:red_eq_total}
\end{equation}
where the first row of Eq. (\ref{eq:red_eq_total}) takes the form
\begin{align}
\dot{s}_1 &= \lambda_1 s_1 + \sum_{i=1}^M\gamma_is_1^{i+1}\bar{s}_1^{i} \\
&+\varepsilon\left(c_{1,\mtrx{0}}\me^{\mi\phi} + \sum_{i=1}^M\left(c_{1,(i,i)}(\Omega)s_1^i\bar{s}_1^i\me^{\mi\phi}+d_{1,(i+1,i-1)}(\Omega)s_1^{i+1}\bar{s}_1^{i-1}\me^{-\mi\phi}\right) \right), \nonumber
\end{align}
Introducing a change to polar coordinates, $s_1=\rho\me^{\mi\theta}$, $\bar{s}_1=\rho\me^{-\mi\theta}$, dividing by $\me^{\mi\theta}$ and introducing the new phase coordinate $\psi = \theta-\phi$, we obtain
\begin{align}
\dot{\rho}+\mi\rho(\dot{\psi}+\Omega) &= \lambda_1 \rho + \sum_{i=1}^M\gamma_i\rho^{2i+1} \label{eq:polar_red_dyn} \\
&+\varepsilon\left(c_{1,\mtrx{0}}\me^{-\mi\psi} + \sum_{i=1}^M\left(c_{1,(i,i)}(\Omega)\rho^{2i}\me^{-\mi\psi}+d_{1,(i+1,i-1)}(\Omega)\rho^{2i}\me^{\mi\psi}\right) \right). \nonumber
\end{align}
We obtain the result listed in Theorem \ref{thm:red_dyn} by splitting Eq. (\ref{eq:polar_red_dyn}) into its real and imaginary part. \qed

\section{Multivariate recurrence relations \label{app:recc}}
\subsection{Products}
The $i$th row on the right hand side of the $\mathcal{O}(1)$ coefficient equation can be written as\begin{equation}
\sum_{j=1}^{2}\partial_{s_j}w_i^0(\mtrx{s})r_j^0(\mtrx{s}) = \sum_{j=1}^{2}\left(\sum_{\substack{\mtrx{m} \\ m_j>0}}m_j W_{i,\mtrx{m}}^0\mtrx{s}^{\mtrx{m}-\mtrx{e}_j}\sum_{\mtrx{n}}R_{j,\mtrx{n}}^0\mtrx{s}^\mtrx{n}\right)
\end{equation}
The $\mtrx{k}$th power coefficient of this resulting product is recursively defined as
\begin{equation}
\left[\sum_{j=1}^{2}\partial_{s_j}w_i^0(\mtrx{s})r_j^0(\mtrx{s})\right]_{\mtrx{k}} = \sum_{j=1}^{2}\sum_{\substack{\mtrx{m}\leq\tilde{\mtrx{k}}_{j} \\ m_j>0}}m_j W_{i,\mtrx{m}}^0R_{j,\tilde{\mtrx{k}}_{j}-\mtrx{m}}^0 \label{eq:kth_coef}
\end{equation}
\begin{example}\label{ex:prod} To demonstrate how the product in Eq. (\ref{eq:kth_coef}) is carried out in \ssm, we assume that we have the following \textit{arbitrary} polynomial functions for the autonomous SSM and autonomous reduced dynamics, which already has been computed up to order $|\mtrx{k}|=3$, where $i=1$,
\begin{align}
w_1^0(\mtrx{s}) = \alpha s_1^3 + \beta s_1^2 s_2,\quad r_1^0(\mtrx{s}) = \gamma s_2^2 + \delta s_1 s_2,\quad r_2^0(\mtrx{s}) = \varepsilon s_2^2, 
\end{align}
We want to compute the coefficient related to the monomial term $\mtrx{k}=(2,2)$, which corresponds to order $|\mtrx{k}|=4$. Using Eq. (\ref{eq:kth_coef}), we write
\begin{equation}
\left[\sum_{j=1}^{2}\partial_{s_j}w_1^0(\mtrx{s})r_j^0(\mtrx{s})\right]_{(2,2)} = \sum_{\substack{\mtrx{m}\leq (3,2) \\ m_1>0}}m_1 W_{1,\mtrx{m}}^0R_{1,(3,2)-\mtrx{m}}^0 + \sum_{\substack{\mtrx{m}\leq (2,3) \\ m_2>0}}m_2 W_{1,\mtrx{m}}^0R_{2,(2,3)-\mtrx{m}}^0 \label{eq:prod_recurrence}
\end{equation}
To increase the efficiency and reduce the total computational time and memory usage, the updated version of \ssm keeps track of all the non-zero coefficients in  $w_1^0(\mtrx{s})$, $r_1^0(\mtrx{s})$ and $r_2^0(\mtrx{s})$. This way, instead of carrying out the full summations in Eq. (\ref{eq:kth_coef}), we can selectively carry out the products from which we know in advance that these terms will give a contribution to the current coefficient of interest. The entries of the non-zero coefficients for each polynomial function are listed in an individual vector and stored in \matlab,
\begin{equation}
W_{1,\text{index}}^0 = 
\begin{bmatrix}
(3,0) \\
(2,1)
\end{bmatrix}, \quad 
R_{1,\text{index}}^0 = 
\begin{bmatrix}
(0,2) \\
(1,1)
\end{bmatrix}, \quad 
R_{2,\text{index}}^0 = 
\begin{bmatrix}
(0,2)
\end{bmatrix}.
\end{equation}
From this we conclude that for the first summation term on the right hand side of Eq. (\ref{eq:prod_recurrence}), the absolute maximum number of iterations that we possibly have to perform are two, related to the terms $\mtrx{m}=(3,0)$ and $\mtrx{m}=(2,1)$, as these are the only currently non-zero terms in $w_1^0(\mtrx{s})$. Depending on the non-zero coefficients of the reduced dynamics, the number of iterations needed either remains the same or decreases. The coefficients, related to $r_1^0(\mtrx{s})$, that are needed in the summation are 
\begin{equation}
R_{1,(3,2)-(3,0)}^0 = R_{1,(0,2)}^0,\quad R_{1,(3,2)-(2,1)}^0 = R_{1,(1,1)}^0,
\end{equation}
which both are non-zero in this particular example. Therefore, we can write 
\begin{align}
\sum_{\substack{\mtrx{m}\leq (3,2) \\ m_1>0}}m_1 W_{1,\mtrx{m}}^0R_{1,(3,2)-\mtrx{m}}^0 = 3W_{1,(3,0)}^0R_{1,(0,2)}^0+2W_{1,(2,1)}^0R_{1,(1,1)}^0 = 3\alpha\gamma + 2\beta\delta.
\end{align}
For the second summation term on the right hand side of Eq. (\ref{eq:prod_recurrence}), the maximum number of iterations that we possibly have to perform is one, corresponding to $\mtrx{m}=(2,1)$, as it is required that $m_2>0$, which is not the case for $\mtrx{m}=(3,0)$. Again, depending on the coefficients of the reduced dynamics, it is possible that less iterations are needed. The coefficients, related to $r_2^0(\mtrx{s})$, that are needed in the summation are 
\begin{equation}
R_{2,(2,3)-(2,1)}^0 = R_{2,(0,2)}^0,
\end{equation}
which is non-zero in this particular example. We can express the second summation term on the right hand side of Eq. (\ref{eq:prod_recurrence}) as
\begin{equation}
 \sum_{\substack{\mtrx{m}\leq (2,3) \\ m_2>0}}m_2 W_{1,\mtrx{m}}^0R_{2,(2,3)-\mtrx{m}}^0 = W_{1,(2,1)}^0R_{2,(0,2)}^0 = \beta\varepsilon.
\end{equation}
Therefore, the coefficient related to the term $\mtrx{k}=(2,2)$ of the product $\sum_{j=1}^{2}\partial_{s_j}w_1^0(\mtrx{s})r_j^0(\mtrx{s})$, is equal to 
\begin{equation}
\left[\sum_{j=1}^{2}\partial_{s_j}w_1^0(\mtrx{s})r_j^0(\mtrx{s})\right]_{(2,2)} = 3\alpha\gamma + 2\beta\delta + \beta\varepsilon.
\end{equation}
For verification, we manually compute the product
\begin{equation}
\sum_{j=1}^{2}\partial_{s_j}w_1^0(\mtrx{s})r_j^0(\mtrx{s}) = \left(3\alpha\gamma+2\beta\delta+\beta\varepsilon\right)s_1^2 s_2^2 + \mathcal{O}(|\mtrx{s}|^4).
\end{equation}
which agrees with our result. 
\end{example}
\subsection{Compositions}
The $i$th row of the composition on the left hand side of Eq. (\ref{eq:auto_SSM}) can be written as 
\begin{equation}
h(\mtrx{s})_a = \sum_{\mtrx{k}}H_{a,\mtrx{k}}\mtrx{s}^\mtrx{k} = (w_i^0(\mtrx{s}))^a=\left(\sum_{\mtrx{m}}W_{i,\mtrx{m}}^0\mtrx{s}^\mtrx{m}\right)^a. \label{eq:rec_comp_init}
\end{equation}
We want to obtain the coefficient related to the term $\mtrx{k}\neq0$ of this composition. We pick an index $j$, such that $k_j=\text{min}(k_l:k_l\neq 0)$ and differentiate Eq. (\ref{eq:rec_comp_init}) with respect to $s_{j}$, yielding
\begin{equation}
\partial_{s_{j}} h(\mtrx{s}) = a(w_i^0(\mtrx{s}))^{a-1}\partial_{s_{j}}w_i^0(\mtrx{s})=ah(\mtrx{s})_{a-1}\partial_{s_{j}}w_i^0(\mtrx{s}),
\end{equation}
which is equivalent to
\begin{equation}
\sum_{\substack{\mtrx{k}\\k_j>0}}k_jH_{a,\mtrx{k}}\mtrx{s}^{\mtrx{k}-\mtrx{e}_j} = a\sum_{\mtrx{n}}H_{a-1,\mtrx{n}}\mtrx{s}^\mtrx{n}\sum_{\substack{\mtrx{m} \\ m_j>0}}m_j W_{i,\mtrx{m}}^0\mtrx{s}^{\mtrx{m}-\mtrx{e}_j}. \label{eq:rec_comp}
\end{equation}
Collecting the coefficient corresponding to the monomial term $\mtrx{s}^{\mtrx{k}-\mtrx{e}_j}$ on each side of Eq. (\ref{eq:rec_comp}) yields the coefficient related to the $\mtrx{k}\neq 0$ term of Eq. (\ref{eq:rec_comp_init}), 
\begin{equation}
H_{a,\mtrx{k}}=\frac{a}{k_j}\sum_{\substack{\mtrx{m}\leq\mtrx{k} \\ m_j>0}}m_jW_{i,\mtrx{m}}^0H_{a-1,\mtrx{k}-\mtrx{m}}. \label{eq:comp_result}
\end{equation}
\begin{example}
We give an demonstration of Eq (\ref{eq:comp_result}), where we will use the same polynomial function $w_1^0(\mtrx{s})$ as in Example \ref{ex:prod},
\begin{align}
w_1^0(\mtrx{s}) = \alpha s_1^3 + \beta s_1^2 s_2.
\end{align}
Assume we are interested in the coefficient related to the monomial term $\mtrx{k}=(5,1)$ of the square of $w_1^0(\mtrx{s})$, i.e. where $a=2$. We choose $j=2$ such that we minimize the number of iterations needed. Then using Eq. (\ref{eq:comp_result}) we can write
\begin{equation}
H_{2,(5,1)}=\frac{2}{1}\sum_{\substack{\mtrx{m}\leq (5,1) \\ m_2>0}}m_2W_{1,\mtrx{m}}^0H_{1,(5,1)-\mtrx{m}},
\end{equation}
where we note that $H_{1,\mtrx{m}}$ is equal to $W_{1,\mtrx{m}}^0$. The entries of the non-zero coefficients for $w_1^0(\mtrx{s})$ are listed in an individual vector,
\begin{equation}
W_{1,\text{index}}^0 = 
\begin{bmatrix}
(3,0) \\
(2,1)
\end{bmatrix}.
\end{equation}
From this we conclude that the absolute maximum number of iterations that we possibly have to perform are two, related to the terms $\mtrx{m}=(3,0)$ and $\mtrx{m}=(2,1)$, as these are the only currently non-zero terms in $w_1^0(\mtrx{s})$. However, taking a closer look, we obverse that for $\mtrx{m}=(3,0)$, $m_2=0$, and therefore this index is excluded from the summation. Summing over the remaining index $\mtrx{m}=(2,1)$, we obtain
\begin{equation}
H_{2,(5,1)}=\frac{2}{1}\sum_{\substack{\mtrx{m}\leq (5,1)\\ m_2>0}}m_2W_{1,\mtrx{m}}^0W_{1,(5,1)-\mtrx{m}}^0 = 2\alpha\beta.
\end{equation}
To verify this result, we manually compute the square of $w_1^0(\mtrx{s})$,
\begin{equation}
(w_1^0(\mtrx{s}))^2 = 2\alpha\beta s_1^5 s_2 + \mathcal{O}(|\mtrx{s}|^6).
\end{equation}
\end{example}

\section{A geometric interpretation of the fixed points of the reduced dynamics  \label{app:geometric}}
We can interpret the zero problem (\ref{eq:zeroproblem}) in a geometric way by multiplying $F_1(\mtrx{u})$ and $F_2(\mtrx{u})$ with  $g_1\neq 0$ and  $f_2\neq 0$, respectively,  and rewriting the result as
\begin{align}
\mtrx{s}(\rho,\Omega,\psi)
&=
\underbrace{\begin{bmatrix}
\cos(\psi)& \sin(\psi)\\
-\sin(\psi) & \cos(\psi)
\end{bmatrix}}_{\mtrx{R}(\psi)}
\underbrace{\begin{bmatrix}
f_2 g_2\\
f_2 g_1
\end{bmatrix}}_{\mtrx{v}_1} +
\underbrace{
\begin{bmatrix}
f_1 g_1-f_2 g_2 \\
0
\end{bmatrix}}_{\mtrx{v}_2} \cos(\psi) \label{eq:zero_rot} \\
&=
\underbrace{-\frac{1}{\varepsilon}
\begin{bmatrix}
g_1 a\\
f_2 (b-\Omega)\rho
\end{bmatrix}}_{\mtrx{v}_3}, \nonumber
\end{align}
where we introduced the rotation matrix $\mtrx{R}(\psi)\in\text{SO}(2)$. For a fixed value of $\rho_0$, $\Omega_0$ and $0\leq\psi<2\pi$, $\mtrx{s}(\rho_0,\Omega_0,\psi)$ represents an ellipse with semi-major and semi-minor axes, $\norm{\mtrx{s}(\rho_0,\Omega_0,\psi_1)}$ and $\norm{\mtrx{s}(\rho_0,\Omega_0,\psi_2)}$, respectively, where 
\begin{equation}
\psi_1 = \argmax{0\leq\psi\leq\pi}\norm{\mtrx{s}(\rho_0,\Omega_0,\psi)},\quad 
\psi_2 = \argmin{0\leq\psi\leq\pi}\norm{\mtrx{s}(\rho_0,\Omega_0,\psi)}. \nonumber
\end{equation}
We can always solve Eq. (\ref{eq:zero_rot}) by scaling  the length of $\mtrx{v}_3$ (varying $\varepsilon$) such that $\mtrx{v}_3$ points to a point on the ellipse $\vec{s}(\rho_0, \Omega_0, \psi)$. This intersection point then defines a $\psi$ value for which Eq. (\ref{eq:zero_rot}) is satisfied. Each point where $\mtrx{s}$ and $\mtrx{v}_3$ coincide for different values of $\rho$ gives a point on the forced-response curve. An illustration of this concept is shown in Fig. \ref{fig:ellipse_1}, where $\mtrx{v}_3$ intersects $\mtrx{s}$ a total of three times for increasing $\rho$. These three intersections correspond to three points on the forced-response curve for a fixed forcing frequency $\Omega$ and fixed forcing amplitude $\varepsilon$.

\begin{figure}[h!]
\centering
    \begin{subfigure}{0.45\textwidth}
        \centering
        \if\mycmd1
        \includegraphics[scale=0.33]{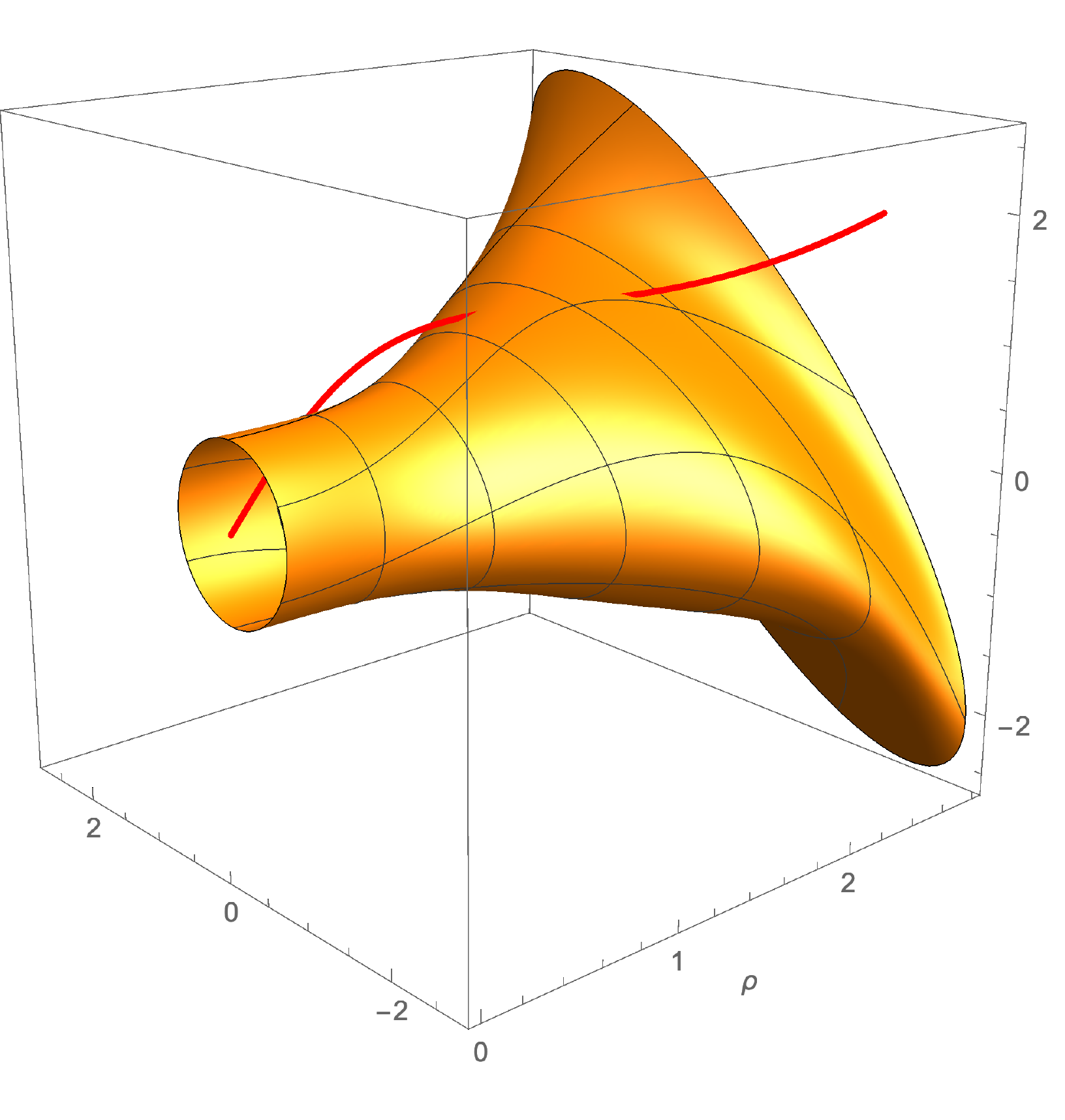}
        \fi
    \end{subfigure}
    \begin{subfigure}{0.45\textwidth}
        \centering
        \if\mycmd1
        \includegraphics[scale=0.31]{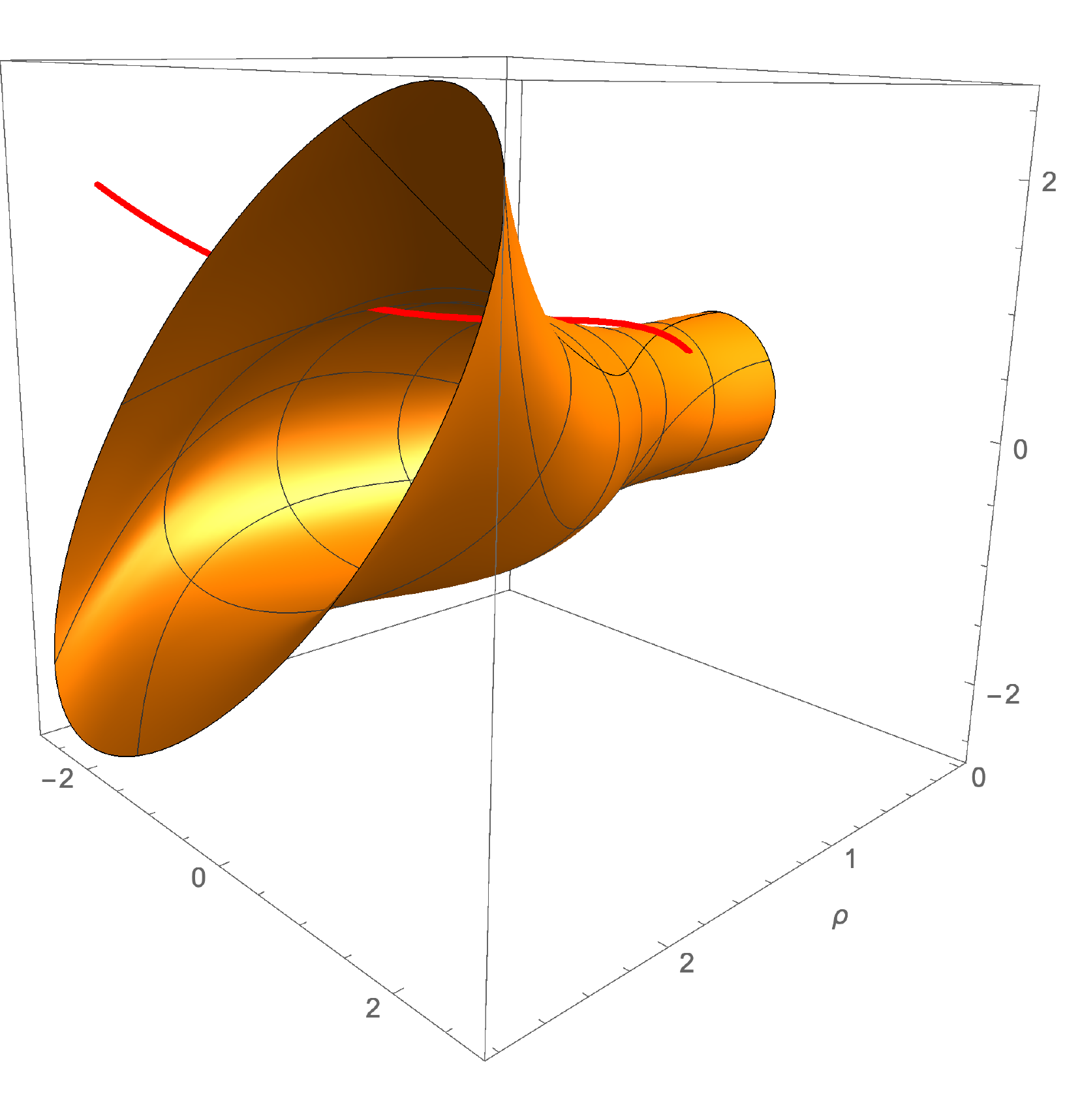}
        \fi
    \end{subfigure}
    \caption{Illustration of $\mtrx{s}(\rho,\Omega,\psi)$ and $\mtrx{v}_3$ for a fixed forcing frequency $\Omega$, $\psi\in[0, 2\pi)$, while varying $\rho$. The points where $\mtrx{s}(\rho,\Omega,\psi)$ and $\mtrx{v}_3$ coincide for different values of $\rho$ will each correspond to a point on the forced-response curve. \label{fig:ellipse_1}}
\end{figure}

We will show that for a mechanical system with symmetric system matrices and with structural damping, we can always pick a modal transformation matrix $\mtrx{T}$, such that $g_1$ and $f_2$ will have a non-zero constant part. 

As seen in Eq. (\ref{eq:c0}), the zeroth-order constant, $c_{1,\mtrx{0}}$, is equal to the first element of the vector $\tilde{\mtrx{F}}_\mtrx{0}/2$, which is extracted from the modal force vector
\begin{equation}
\mtrx{F}_\text{m}(\phi) = \tilde{\mtrx{T}}^{-1}
\begin{bmatrix}
\mtrx{0}\\
\mtrx{M}^{-1}\mtrx{f}(\phi)
\end{bmatrix}=\frac{\tilde{\mtrx{F}}_{\mtrx{0}}}{2}\left(\me^{\mi\phi}+\me^{-\mi\phi}\right). \label{eq:modalforce}
\end{equation}
For a mechanical system with symmetric system matrices and with structural damping, following \cite{Breunung2017}, we introduce a mass normalized real modal transformation matrix $\mtrx{E}$, defined in terms of the quantities in the second-order system (\ref{eq:mech_sys}) as follows:
\begin{gather}
(\mtrx{M}^{-1}\mtrx{K})\mtrx{E}=\mtrx{E}\text{ diag}(\omega_1^2,\ldots,\omega_n^2), \nonumber \\
\mtrx{E}^{\top}\mtrx{M}\mtrx{E}=\mtrx{I},\quad \mtrx{E}^{\top}\mtrx{C}\mtrx{E}=\text{diag}(\beta_1,\ldots,\beta_n),\quad \mtrx{E}^{\top}\mtrx{K}\mtrx{E}=\text{diag}(\omega_1^2,\ldots,\omega_n^2). \nonumber
\end{gather}
Here the eigenvalues of the linearized part of system (\ref{eq:dyn_sys}) are given by 
\begin{equation}
\lambda_{2i-1} = -\frac{\beta_i}{2}+\sqrt{\left(\frac{\beta_i}{2}\right)^2-\omega_i^2},\quad \lambda_{2i} = -\frac{\beta_i}{2}-\sqrt{\left(\frac{\beta_i}{2}\right)^2-\omega_i^2},\quad i=1,\ldots,n.
\end{equation}
We now introduce the modal transformation matrix $\hat{\mtrx{T}}$ that will diagonalize the linear matrix $\mtrx{A}$ in (\ref{eq:dyn_sys}), i.e., we let
\begin{gather}
\hat{\mtrx{T}}=
\begin{bmatrix}
\mtrx{E} & \mtrx{E} \\
\mtrx{E}\gmtrx{\Lambda}_1& \mtrx{E}\gmtrx{\Lambda}_2
\end{bmatrix},\quad \hat{\gmtrx{\Lambda}}=\hat{\mtrx{T}}^{-1}\mtrx{A}\hat{\mtrx{T}}=
\begin{bmatrix}
\gmtrx{\Lambda}_1 & \mtrx{0} \\
\mtrx{0} & \gmtrx{\Lambda}_2
\end{bmatrix}\\
\gmtrx{\Lambda}_1 = \text{diag}(\lambda_1,\lambda_3,\ldots,\lambda_{2n-1}),\quad \gmtrx{\Lambda}_2=\text{diag}(\lambda_2,\lambda_4,\ldots,\lambda_{2n}) = \bar{\gmtrx{\Lambda}}_1. \nonumber
\end{gather}
The inverse of the modal transformation matrix $\hat{\mtrx{T}}$ is given by
\begin{gather}
\hat{\mtrx{T}}^{-1}=
\begin{bmatrix}
\mtrx{E}^{-1}+(\mtrx{\Lambda}_2-\mtrx{\Lambda}_1)^{-1}\mtrx{E}^{-1}\gmtrx{\Lambda}_1\mtrx{E}^{-1} & -(\mtrx{\Lambda}_2-\mtrx{\Lambda}_1)^{-1}\mtrx{E}^{-1} \\
(\mtrx{\Lambda}_2-\mtrx{\Lambda}_1)^{-1}\mtrx{E}^{-1}\gmtrx{\Lambda}_1\mtrx{E}^{-1} & (\mtrx{\Lambda}_2-\mtrx{\Lambda}_1)^{-1}\mtrx{E}^{-1}
\end{bmatrix}.
\end{gather}
We observe that the last $n$ columns of $\hat{\mtrx{T}}^{-1}$ are purely imaginary. Note that the current ordering of the columns of $\hat{\mtrx{T}}$, will result in a diagonalized matrix $\hat{\gmtrx{\Lambda}}$ with a different column ordering as compared to $\gmtrx{\Lambda}$ in (\ref{eq:ds_diag}). However, we can always reorder the columns of $\hat{\mtrx{T}}$ to $\tilde{\mtrx{T}}$ such that we obtain the original diagonalized matrix $\gmtrx{\Lambda}$, without altering the fact that the last $n$ columns of $\tilde{\mtrx{T}}^{-1}$ will be imaginary. This is due to the fact that a reordering of the columns of a full rank matrix $\mtrx{P}$ will result in a reordering of the rows of $\mtrx{P}^{-1}$, but not the columns of $\mtrx{P}^{-1}$. 

As a result, the vector $\tilde{\mtrx{F}}_{\mtrx{0}}$ will be purely imaginary as can be seen from Eq. (\ref{eq:modalforce}), and, consequently, the zeroth order constant $c_{1,\mtrx{0}}$ in (\ref{eq:c0}) will be purely imaginary. Additionally, the first $n$ rows of $\tilde{\mtrx{T}}$ are real (as $\tilde{\mtrx{T}}$ is only a column shifted version of $\hat{\mtrx{T}}$), meaning that the if we map a fixed point for the reduced system back to the full phase space, we observe that the leading order linear term in $\rho$, corresponding to a positional coordinate $y_i$ of the full system,  will have a phase shift of $\psi$ with respect to the forcing, i.e.
\begin{align}
y_i &= [\tilde{\mtrx{T}}]_{i,1}\rho\me^{\mi(\phi+\psi)}+[\tilde{\mtrx{T}}]_{i,2}\rho\me^{-\mi(\phi+\psi)}+ \mathcal{O}(|\rho|^2,\varepsilon) \nonumber\\
&= [\tilde{\mtrx{T}}]_{i,1}\rho\left(\me^{\mi(\phi+\psi)}+\me^{-\mi(\phi+\psi)}\right) + \mathcal{O}(|\rho|^2,\varepsilon),\quad i=1,\ldots,n, \nonumber
\end{align}
provided that $[\tilde{\mtrx{T}}]_{i,1}=[\tilde{\mtrx{T}}]_{i,2}\neq 0$. No additional phase is introduced by the coefficients of the modal transformation matrix for the positional coordinates $y_i$, as all the coefficients are real.

In the setting of Breunung and Haller \cite{Breunung2017}, where the parameterization $\mtrx{W}(\mtrx{s},\phi)$ and the reduced dynamics $\mtrx{R}(\mtrx{s},\phi)$ are truncated at $\mathcal{O}(\varepsilon|\mtrx{s}|,\varepsilon^2)$, which is justified when  $\mtrx{s}=\mathcal{O}(\varepsilon^{\frac{1}{2M+2}})$, the zero problem (\ref{eq:zeroproblem}) can be written as
\begin{equation}
\tilde{\mtrx{F}}(\mtrx{u}) = 
\begin{bmatrix}
\tilde{F}_1(\mtrx{u})\\
\tilde{F}_2(\mtrx{u})
\end{bmatrix} = 
\begin{bmatrix}
a(\rho) + \varepsilon\left(\text{Re}(c_{1,\mtrx{0}})\cos(\psi)+\text{Im}(c_{1,\mtrx{0}})\sin(\psi)\right) \\
(b(\rho)-\Omega)\rho + \varepsilon\left(\text{Im}(c_{1,\mtrx{0}})\cos(\psi)-\text{Re}(c_{1,\mtrx{0}})\sin(\psi)\right) 
\end{bmatrix} = \mtrx{0}.
\end{equation}
The ellipse $\mtrx{s}$ reduces to a circle
\begin{equation}
\underbrace{\begin{bmatrix}
\cos(\psi)& \sin(\psi)\\
-\sin(\psi) & \cos(\psi)
\end{bmatrix}}_{\mtrx{R}(\psi)}
\underbrace{\begin{bmatrix}
\text{Re}(c_{1,\mtrx{0}}) \\
\text{Im}(c_{1,\mtrx{0}})
\end{bmatrix}}_{\mtrx{v}_1}
=
\underbrace{-\frac{1}{\varepsilon}
\begin{bmatrix}
a(\rho)\\
(b(\rho)-\Omega)\rho
\end{bmatrix}}_{\mtrx{v}_2}. \label{eq:zero_rot_0}
\end{equation}
In their setting, at the intersection of the FRC with the autonomous backbone curve, i.e., where $b(\rho)-\Omega=0$, the vectors $\mtrx{v}_1$ and $\mtrx{v}_2$ are orthogonal with respect to each other, due to the fact the real part of $c_{1,\mtrx{0}}$ is zero. Therefore, the phase shift $\psi$ will be equal to $\pi/2$.
\bibliographystyle{unsrt} 
\bibliography{ref.bib}

\end{document}